\documentclass{article}

\usepackage{titling}
\thanksmarkseries{alph}
\makeindex

\usepackage{amsmath, amsthm, amssymb, mathtools}
\usepackage{thmtools} 
\usepackage{standalone}
\usepackage{tikz}
\usetikzlibrary{shapes.geometric} 

\usepackage{ytableau}
\usepackage{subcaption}
\usepackage{booktabs}
\usepackage{hyperref}
\hypersetup{colorlinks=true, linkcolor=black, urlcolor=black, linktoc=all, citecolor=black}
\usepackage{cleveref}

\newtheorem{theorem}{Theorem}[section]
\newtheorem{lemma}[theorem]{Lemma}
\newtheorem{corollary}{Corollary}[theorem]

\newcommand{\abs}[1]{\left| #1 \right|}
\newcommand{\floor}[1]{\left\lfloor #1 \right\rfloor}
\newcommand{\ceil}[1]{\left\lceil #1 \right\rceil}

\newcommand{\bfb}[1]{{\textbf{#1}}}

\newcommand{\poly}{\gamma}

\newcommand{\chartpoly}[5]{
        \begin{scope}[xshift = #1 cm, yshift = #2 cm]
        \draw[ultra thick] #5;
        \node at (#3, -.5) {#4};
        \clip #5;
        \foreach \z in {0,1,2,3,4,5}{
            \draw (\z,0) --++ (0,5);
            \draw (0,\z) --++ (5,0);
        }
    \end{scope}}

\newcommand{\shadepoly}[5]{
        \begin{scope}[xshift = #1 cm, yshift = #2 cm]
        \draw[-] #5;
        \node at (#3, -.5) {#4};
        \clip #5;
        \draw[draw=black!10, fill=black!10] (0,0) rectangle (5,5);
        \foreach \z in {0,1,2,3,4,5}{
            \draw (\z,0) --++ (0,5);
            \draw (0,\z) --++ (5,0);
        }
    \end{scope}}
        
\newcommand{\zig}[1]{
\begin{tikzpicture}[scale=1/#1]
    \foreach \z in {1,2,...,#1}{
    \draw[-] (0,0) rectangle (\z,#1+1-\z);
    }
\end{tikzpicture}
}

\newcommand{\polygonalarray}[4]{
        \pgfmathtruncatemacro{\n}{#1-1}
        \draw[-] (#3,0)
        \foreach \x in {0,...,\n}{
        \foreach \y in {1,...,#2}{
        --++ (\x*360/#1:#4*1cm) node [shape=circle, fill=black,scale=.45]{}
        }
        }
        -- cycle;
}

\newcommand{\Aztecpath}[1]{
    --++(1,0)
    \foreach \z in {1,...,#1}{
    --++(0,1)--++(1,0)
    }
    --++(0,1)
    \foreach \z in {1,...,#1}{
    --++(-1,0)--++(0,1)
    }
    --++(-1,0)
    \foreach \z in {1,...,#1}{
    --++(0,-1)--++(-1,0)
    }
    --++(0,-1)
    \foreach \z in {1,...,#1}{
    --++(1,0)--++(0,-1)
    }
    --cycle
}

\newcommand{\Ashape}{\draw[fill=black!10] (0,1) \Aztecpath{3};
        \foreach \z in {0,1,2}{
            \draw (-\z, 2+\z) --++ (0,5-2*\z); 
            \draw (1+\z, 2+\z) --++ (0,5-2*\z); 
            
            \draw (-\z, 2+\z) --++ (1+2*\z,0);
            \draw (-\z, 7-\z) --++ (1+2*\z,0);  
}}

\newcommand{\Fmopath}[1]{
    --++(1,0)
    \foreach \z in {1,...,#1}{
    --++(0,2)--++(1,0)
    }
    --++(0,1)--++(-1-#1,0)
    \foreach \z in {1,...,#1}{
    --++(0,-1)--++(-1,0)
    }
    --++(0,-1)
    \foreach \z in {1,...,#1}{
    --++(1,0)--++(0,-1)
    }
    --cycle
}

\newcommand{\cell}[2]{\draw[-] (#1,#2) rectangle ++ (1,1);}

\newcommand{\numcell}[3]{\draw[-] (#1,#2) rectangle ++ (1,1);\node at (#1+.5,#2+.5){#3};}

\newcommand{\dmo}{\begin{tikzpicture}[scale=0.25, rotate=90]
        \draw[black,  semithick](0,0)rectangle(0.5,1);
        \draw[black,  semithick] (0,.5) --++(.5,0);
        \end{tikzpicture}}

\newcommand{\tmo}{\begin{tikzpicture}[scale=0.25]
        \draw[black,  semithick](0,0)rectangle(0.5,0.5);
        \draw[black,  semithick](0,0.5)rectangle(0.5,1);
        \draw[black,  semithick](0.5,0)rectangle(1,0.5);
        \end{tikzpicture}}
        
\newcommand{\stmo}{\begin{tikzpicture}[scale=0.25]
        \draw[black,  semithick](0,0)rectangle++(0.5,0.5);
        \draw[black,  semithick](.5,0)rectangle++(0.5,0.5);
        \draw[black,  semithick](1,0)rectangle++(0.5,0.5);
        \end{tikzpicture}}

\newcommand{\sqmo}{\begin{tikzpicture}[scale=0.25]
        \draw[black,  semithick](0,0)rectangle(0.5,0.5);
        \draw[black,  semithick](0,0.5)rectangle(0.5,1);
        \draw[black,  semithick](0.5,0)rectangle(1,0.5);
        \draw[black,  semithick](0.5,0.5)rectangle(1,1);
    \end{tikzpicture}}

\newcommand{\lmo}{\begin{tikzpicture}[scale=0.25]
        \draw[black,  semithick](0,0)rectangle(0.5,0.5);
        \draw[black, semithick](0,0.5)rectangle(0.5,1);
        \draw[black,  semithick](0.5,0)rectangle(1,0.5);
        \draw[black,  semithick](1,0)rectangle(1.5,0.5);
        \end{tikzpicture}}  

\newcommand{\ttmo}{\begin{tikzpicture}[scale=0.25]
        \draw[black,  semithick](-0.5,0)rectangle++(0.5,0.5);
        \draw[black,  semithick](0,0)rectangle++(0.5,0.5);
        \draw[black,  semithick](0,0.5)rectangle++(0.5,0.5);
        \draw[black,  semithick](0.5,0)rectangle++(0.5,0.5);
        \end{tikzpicture}}

\newcommand{\zmo}{\begin{tikzpicture}[scale=0.25]
        \draw[black,  semithick](0,0)rectangle(0.5,0.5);
        \draw[black,  semithick](-0.5,0.5)rectangle(0,1);
        \draw[black,  semithick](0,0.5)rectangle(0.5,1);
        \draw[black,  semithick](0.5,0)rectangle(1,0.5);
        \end{tikzpicture}}

\newcommand{\tmot}{\begin{tikzpicture}[scale=0.16]
        \draw[black](0,0)rectangle(0.5,0.5);
        \draw[black](0,0.5)rectangle(0.5,1);
        \draw[black](0.5,0)rectangle(1,0.5);
        \end{tikzpicture}}

\newcommand{\lmot}{\begin{tikzpicture}[scale=0.16]
        \draw[black](0,0)rectangle(0.5,0.5);
        \draw[black](0,0.5)rectangle(0.5,1);
        \draw[black](0.5,0)rectangle(1,0.5);
        \draw[black](1,0)rectangle(1.5,0.5);
        \end{tikzpicture}} 

\newcommand{\sqmot}{\begin{tikzpicture}[scale=0.16]
        \draw[black](0,0)rectangle(0.5,0.5);
        \draw[black](0,0.5)rectangle(0.5,1);
        \draw[black](0.5,0)rectangle(1,0.5);
        \draw[black](0.5,0.5)rectangle(1,1);
    \end{tikzpicture}}

\newcommand{\tttmo}{\begin{tikzpicture}[scale=0.16]
        \draw[black](-0.5,0)rectangle(0,0.5);
        \draw[black](0,0)rectangle(0.5,0.5);
        \draw[black](0,0.5)rectangle(0.5,1);
        \draw[black](0.5,0)rectangle(1,0.5);
        \end{tikzpicture}}
        
\newcommand{\tumo}{\begin{tikzpicture}[scale=0.16]
        \draw[black](0,0)rectangle(0.5,0.5);
        \draw[black](0,0.5)rectangle(0.5,1);
        \draw[black](0.5,0)rectangle(1,0.5);
        \draw[black](1,0)rectangle(1.5,0.5);
        \draw[black](1,0.5)rectangle(1.5,1);
        \end{tikzpicture}}

\newcommand{\zmot}{\begin{tikzpicture}[scale=0.16]
        \draw[black](0,0)rectangle(0.5,0.5);
        \draw[black](-0.5,0.5)rectangle(0,1);
        \draw[black](0,0.5)rectangle(0.5,1);
        \draw[black](0.5,0)rectangle(1,0.5);
        \end{tikzpicture}}

\newcommand{\fmo}{\begin{tikzpicture}[scale=0.16]
        \draw[black,  semithick](0,0)rectangle(0.5,0.5);
        \draw[black,  semithick](0,0.5)rectangle(0.5,1);
        \draw[black,  semithick](0.5,1)rectangle(1,1.5);
        \draw[black,  semithick](0,0.5)rectangle(-0.5,1);
        \draw[black,  semithick](0,1)rectangle(0.5,1.5);
        \end{tikzpicture}}

\newcommand{\lfmo}{\begin{tikzpicture}[scale=0.25]
        \draw[black,  semithick](0,0)rectangle(0.5,0.5);
        \draw[black,  semithick](0,0.5)rectangle(0.5,1);
        \draw[black,  semithick](0.5,1)rectangle(1,1.5);
        \draw[black,  semithick](0,0.5)rectangle(-0.5,1);
        \draw[black,  semithick](0,1)rectangle(0.5,1.5);
        \end{tikzpicture}}

\newcommand{\lpmo}{\begin{tikzpicture}[scale=0.25]
        \draw[black, semithick](0,0)rectangle(0.5,0.5);
        \draw[black, semithick](0,0.5)rectangle(0.5,1);
        \draw[black, semithick](0.5,0)rectangle(1,0.5);
        \draw[black, semithick](1,0)rectangle(1.5,0.5);
        \draw[black, semithick](1.5,0)rectangle(2,0.5);
        \end{tikzpicture}}  

\newcommand{\smo}{\begin{tikzpicture}[scale=.25]
        \draw[black,  semithick](0,0)rectangle(0.5,0.5);
        \draw[black,  semithick](0.5,0)rectangle(1,0.5);
        \draw[black,  semithick](1,0)rectangle(1.5,0.5);
        \draw[black,  semithick](0,0.5)rectangle(0.5,1);
        \draw[black,  semithick](-0.5,0.5)rectangle(0,1);
        \end{tikzpicture}}
        
\newcommand{\ppmo}{\begin{tikzpicture}[scale=0.25,x={(90:1cm)},y={(0:1cm)}]
        \draw[black,  semithick](0,0)rectangle(0.5,0.5);
        \draw[black,  semithick](0,0.5)rectangle(0.5,1);
        \draw[black,  semithick](0.5,0)rectangle(1,0.5);
        \draw[black,  semithick](0.5,0.5)rectangle(1,1);
        \draw[black,  semithick](0,1)rectangle(0.5,1.5);
    \end{tikzpicture}}

\newcommand{\tpmo}{\begin{tikzpicture}[scale=0.16]
        \draw[black,  semithick](0,0)rectangle++(0.5,0.5);
        \draw[black,  semithick](0,0.5)rectangle++(0.5,0.5);
        \draw[black,  semithick](0.5,0)rectangle++(0.5,0.5);
        \draw[black,  semithick](-.5,0)rectangle++(0.5,0.5);
        \draw[black,  semithick](0,1)rectangle++(0.5,0.5);
        \end{tikzpicture}} 

\newcommand{\ltpmo}{\begin{tikzpicture}[scale=0.25]
        \draw[black,  semithick](0,0)rectangle++(0.5,0.5);
        \draw[black,  semithick](0,0.5)rectangle++(0.5,0.5);
        \draw[black,  semithick](0.5,0)rectangle++(0.5,0.5);
        \draw[black,  semithick](-.5,0)rectangle++(0.5,0.5);
        \draw[black,  semithick](0,1)rectangle++(0.5,0.5);
        \end{tikzpicture}}

\newcommand{\umo}{\begin{tikzpicture}[scale=0.25]
        \draw[black,  semithick](0,0)rectangle(0.5,0.5);
        \draw[black,  semithick](0,0.5)rectangle(0.5,1);
        \draw[black,  semithick](0.5,0)rectangle(1,0.5);
        \draw[black,  semithick](1,0)rectangle(1.5,0.5);
        \draw[black,  semithick](1,0.5)rectangle(1.5,1);
        \end{tikzpicture}}

\newcommand{\vpmo}{\begin{tikzpicture}[scale=0.16]
        \draw[black,  semithick](0,0)rectangle(0.5,0.5);
        \draw[black,  semithick](0,0.5)rectangle(0.5,1);
        \draw[black,  semithick](0.5,0)rectangle(1,0.5);
        \draw[black,  semithick](1,0)rectangle(1.5,0.5);
        \draw[black,  semithick](0,1)rectangle(0.5,1.5);
        \end{tikzpicture}}  

\newcommand{\lvpmo}{\begin{tikzpicture}[scale=0.25]
        \draw[black,  semithick](0,0)rectangle(0.5,0.5);
        \draw[black,  semithick](0,0.5)rectangle(0.5,1);
        \draw[black,  semithick](0.5,0)rectangle(1,0.5);
        \draw[black,  semithick](1,0)rectangle(1.5,0.5);
        \draw[black,  semithick](0,1)rectangle(0.5,1.5);
        \end{tikzpicture}}  

\newcommand{\wmo}{\begin{tikzpicture}[scale=.16]
        \draw[black,  semithick](1,0)rectangle++(0.5,0.5);
        \draw[black,  semithick](0.5,0)rectangle++(0.5,0.5);
        \draw[black,  semithick](0.5,0.5)rectangle++(0.5,0.5);
        \draw[black,  semithick](0,0.5)rectangle++(0.5,0.5);
        \draw[black,  semithick](0,1)rectangle++(0.5,0.5);
        \end{tikzpicture}}

\newcommand{\lwmo}{\begin{tikzpicture}[scale=.25]
        \draw[black,  semithick](1,0)rectangle++(0.5,0.5);
        \draw[black,  semithick](0.5,0)rectangle++(0.5,0.5);
        \draw[black,  semithick](0.5,0.5)rectangle++(0.5,0.5);
        \draw[black,  semithick](0,0.5)rectangle++(0.5,0.5);
        \draw[black,  semithick](0,1)rectangle++(0.5,0.5);
        \end{tikzpicture}}
        
\newcommand{\xmo}{\begin{tikzpicture}[scale=.16]
        \draw[black,  semithick](0,0)rectangle(0.5,0.5);
        \draw[black,  semithick](0,0.5)rectangle(0.5,1);
        \draw[black,  semithick](0,1)rectangle(0.5,1.5);
        \draw[black,  semithick](0.5,0.5)rectangle(1,1);
        \draw[black,  semithick](-0.5,0.5)rectangle(0,1);
        \end{tikzpicture}}

\newcommand{\lxmo}{\begin{tikzpicture}[scale=.25]
        \draw[black,  semithick](0,0)rectangle(0.5,0.5);
        \draw[black,  semithick](0,0.5)rectangle(0.5,1);
        \draw[black,  semithick](0,1)rectangle(0.5,1.5);
        \draw[black,  semithick](0.5,0.5)rectangle(1,1);
        \draw[black,  semithick](-0.5,0.5)rectangle(0,1);
        \end{tikzpicture}}

\newcommand{\ypmo}{\begin{tikzpicture}[scale=0.25]
        \draw[black,  semithick](0,0)rectangle(0.5,0.5);
        \draw[black,  semithick](1,0.5)rectangle(0.5,1);
        \draw[black,  semithick](0.5,0)rectangle(1,0.5);
        \draw[black,  semithick](1,0)rectangle(1.5,0.5);
        \draw[black,  semithick](1.5,0)rectangle(2,0.5);
        \end{tikzpicture}} 

\newcommand{\zpmo}{\begin{tikzpicture}[scale=0.16]
        \draw[black,  semithick](0,0)rectangle(0.5,0.5);
        \draw[black,  semithick](0,0.5)rectangle(0.5,1);
        \draw[black,  semithick](0.5,0)rectangle(1,0.5);
        \draw[black,  semithick](0,1)rectangle++(-0.5,0.5);
        \draw[black,  semithick](0,1)rectangle(0.5,1.5);
        \end{tikzpicture}}

\newcommand{\lzpmo}{\begin{tikzpicture}[scale=0.25]
        \draw[black,  semithick](0,0)rectangle(0.5,0.5);
        \draw[black,  semithick](0,0.5)rectangle(0.5,1);
        \draw[black,  semithick](0.5,0)rectangle(1,0.5);
        \draw[black,  semithick](0,1)rectangle++(-0.5,0.5);
        \draw[black,  semithick](0,1)rectangle(0.5,1.5);
        \end{tikzpicture}}  

\newcommand{\zpmobackwards}{\begin{tikzpicture}[xscale=0.16,yscale=-.16]
        \draw[black,  semithick](0,0)rectangle(0.5,0.5);
        \draw[black,  semithick](0,0.5)rectangle(0.5,1);
        \draw[black,  semithick](0.5,0)rectangle(1,0.5);
        \draw[black,  semithick](0,1)rectangle++(-0.5,0.5);
        \draw[black,  semithick](0,1)rectangle(0.5,1.5);
        \end{tikzpicture}}

\newcommand{\tfmo}{\begin{tikzpicture}[scale=0.12]
        \draw[black](0,0)rectangle(0.5,0.5);
        \draw[black](0,0.5)rectangle(0.5,1);
        \draw[black](0.5,1)rectangle(1,1.5);
        \draw[black](0,0.5)rectangle(-0.5,1);
        \draw[black](0,1)rectangle(0.5,1.5);
        \end{tikzpicture}}    

\newcommand{\tlpmo}{\begin{tikzpicture}[scale=0.16]
        \draw[black](0,0)rectangle(0.5,0.5);
        \draw[black](0,0.5)rectangle(0.5,1);
        \draw[black](0.5,0)rectangle(1,0.5);
        \draw[black](1,0)rectangle(1.5,0.5);
        \draw[black](1.5,0)rectangle(2,0.5);
        \end{tikzpicture}}    
        
\newcommand{\tppmo}{\begin{tikzpicture}[scale=0.16,x={(90:1cm)},y={(0:1cm)}]
        \draw[black](0,0)rectangle(0.5,0.5);
        \draw[black](0,0.5)rectangle(0.5,1);
        \draw[black](0.5,0)rectangle(1,0.5);
        \draw[black](0.5,0.5)rectangle(1,1);
        \draw[black](0,1)rectangle(0.5,1.5);
    \end{tikzpicture}}
    
\newcommand{\ttpmo}{\begin{tikzpicture}[scale=0.12]
        \draw[black](0,0)rectangle++(0.5,0.5);
        \draw[black](0,0.5)rectangle++(0.5,0.5);
        \draw[black](0.5,0)rectangle++(0.5,0.5);
        \draw[black](-.5,0)rectangle++(0.5,0.5);
        \draw[black](0,1)rectangle++(0.5,0.5);
        \end{tikzpicture}} 

\newcommand{\tvpmo}{\begin{tikzpicture}[scale=0.12]
        \draw[black](0,0)rectangle(0.5,0.5);
        \draw[black](0,0.5)rectangle(0.5,1);
        \draw[black](0.5,0)rectangle(1,0.5);
        \draw[black](1,0)rectangle(1.5,0.5);
        \draw[black](0,1)rectangle(0.5,1.5);
        \end{tikzpicture}}  
        
\newcommand{\txmo}{\begin{tikzpicture}[scale=.12]
        \draw[black](0,0)rectangle(0.5,0.5);
        \draw[black](0,0.5)rectangle(0.5,1);
        \draw[black](0,1)rectangle(0.5,1.5);
        \draw[black](0.5,0.5)rectangle(1,1);
        \draw[black](-0.5,0.5)rectangle(0,1);
        \end{tikzpicture}}

\newcommand{\tzpmo}{\begin{tikzpicture}[scale=0.12]
        \draw[black](0,0)rectangle(0.5,0.5);
        \draw[black](0,0.5)rectangle(0.5,1);
        \draw[black](0.5,0)rectangle(1,0.5);
        \draw[black](0,1)rectangle++(-0.5,0.5);
        \draw[black](0,1)rectangle(0.5,1.5);
        \end{tikzpicture}}

\newcommand{\tphmo}{\begin{tikzpicture}[scale=0.16,x={(90:1cm)},y={(0:1cm)}]
        \draw[black](0,0)rectangle(0.5,0.5);
        \draw[black](0,0.5)rectangle(0.5,1);
        \draw[black](0.5,0)rectangle(1,0.5);
        \draw[black](0.5,0.5)rectangle(1,1);
        \draw[black](0,1)rectangle(0.5,1.5);
        \draw[black](0,1.5)rectangle++(0.5,0.5);
    \end{tikzpicture}}

\newcommand{\tfidget}{    \resizebox{13 pt}{!}{
        \begin{tikzpicture}[scale=0.4]
        \node[ line width= 1.5pt, regular polygon, regular polygon sides=6, minimum size=0.1cm, draw]at (0,0){ };
        \node[ line width= 1.5pt, regular polygon, regular polygon sides=6, minimum size=0.1cm, draw]at (0.75,0.433){ }; 
        \node[ line width= 1.5pt, regular polygon, regular polygon sides=6, minimum size=0.1cm, draw]at (0.75,1.299){ };
        \node[ line width= 1.5pt, regular polygon, regular polygon sides=6, minimum size=0.1cm, draw]at (1.5,0){ };
    \end{tikzpicture}} }

\newcommand{\cthex}{\rotatebox{90}{\resizebox{9 pt}{!}{
\begin{tikzpicture}[scale=0.4]
\node[ line width= 1.5pt, regular polygon, regular polygon sides=6, minimum size=0.1cm, draw]at (0,0){ };
    \node[ line width= 1.5pt, regular polygon, regular polygon sides=6, minimum size=0.1cm, draw]at (0.75,0.433){ }; 
    \node[ line width= 1.5pt, regular polygon, regular polygon sides=6, minimum size=0.1cm, draw]at (0.75,1.299){ };
    \node[ line width= 1.5pt, regular polygon, regular polygon sides=6, minimum size=0.1cm, draw]at (0,1.732){ };
\end{tikzpicture}}}}

\usepackage[backend=biber,
style=alphabetic]{biblatex}
\title{Polyomino Density}
\author{D. M. Condon\thanks{St. Lawrence University (dcondon@stlawu.edu)}
  \and Eli B. Dugan\thanks{Ohio State University (dugan.248@osu.edu)}
  \and Laney M. Goldman\thanks{Georgia Institute of Technology (lgoldman6@gatech.edu)}
  \and Emily R. Williams}
\date{}

\begin{document}
\maketitle

\begin{abstract}
    A de Bruijn polyomino has colored cells and includes exactly one instance of each possible coloring of another polyomino with those colors. This generalizes the notion of a de Bruijn sequence. 
    We are interested in the de Bruijn polyominoes of minimum size.
    As a step toward finding these, we introduce the problem of finding the smallest polyominoes containing at least $N$ translated copies of the polyomino $p$, allowing overlaps. We say these are $(p,N)$-dense and have size $a_{p,N}$. For certain pairs of polyominoes $p$ and $q$ of equal size, we show there are polyomino transformations relating $a_{p,N}$ and $a_{q,N}$. Using these transformations, we identify classes of polyominoes which share the sequence $(a_{p,N})_{N=1}^\infty$. We give closed forms for $(a_{p,N})_{N=1}^\infty$ for most polyominoes with up to five cells. Leveraging these results we introduce novel de Bruijn polyominoes, as well as other colored polyforms with de Bruijn-like properties.
\end{abstract}

\section{Introduction}\label{sec:introduction}

A \bfb{polyomino} is a connected shape made from unit squares, called \bfb{cells}, glued together edge-to-edge. 
The name plays on the word \bfb{domino}, as dominoes are 2-celled polyominoes. Similarly, the terms \bfb{monomino}, \bfb{tromino}, \bfb{tetromino}, \bfb{pentomino}, and \bfb{$n$-omino} refer to polyominoes having one, three, four, five,  and $n$ cells respectively.
\Cref{fig: polyomino chart} depicts the polyominoes with up to five cells, up to congruence. Small polyominoes have standard names, which are included in the figure. 
A polyomino with cells forming a single row or column is called a \bfb{straight} polyomino. The \bfb{size} of a polyomino $p$ is its number of cells, denoted $|p|$. For example, $\big | {\sqmo} \big | = 4$ and $\big | {\ppmo} \big | = 5$.

\begin{figure}
    \centering
    \input{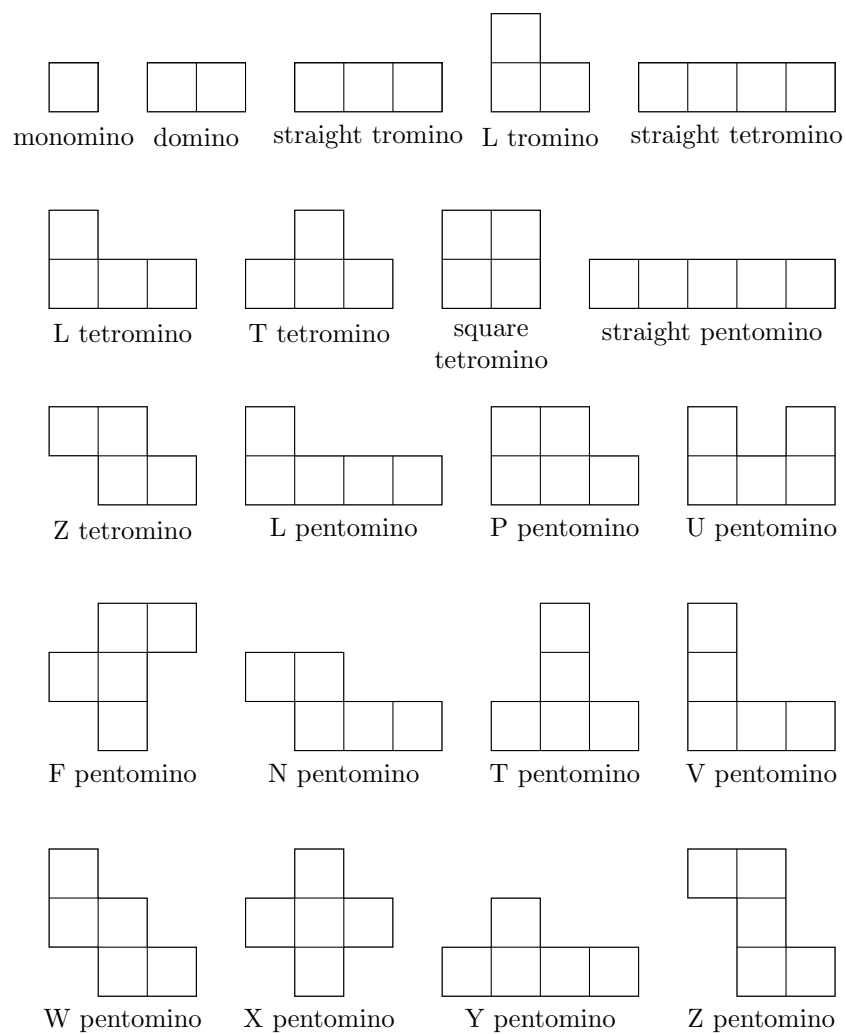}
    \caption{There are two varieties of tromino, five varieties of tetromino, and twelve varieties of pentomino, up to congruence.}
    \label{fig: polyomino chart}
\end{figure}

A \bfb{coloring} of a polyomino by a set of colors $X$ is a way of assigning each cell of the polyomino exactly one color from $X$. For $p$ a polyomino and $X$ a set of $n$ colors, we say a polyomino $P$ colored by $X$ is \bfb{$(p,n)$-de Bruijn} if every coloring of $p$ by $X$ occurs within $P$ exactly once. See \Cref{fig:sixteen squares} as an example.

\begin{figure}
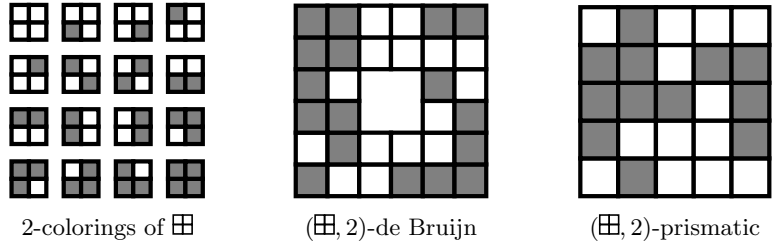

        \centering
    \begin{subfigure}{0.3\textwidth}
        \centering
        \input{Images/colorings}
        \caption*{$2$-colorings of ${\protect\sqmo}$}
        \label{fig: square colorings}
    \end{subfigure}
    \begin{subfigure}{0.3\textwidth}
        \centering
        \input{Images/DeBruijn}
        \caption*{$({\protect\sqmo}, 2)$-de Bruijn}
        \label{fig: square debruijn}
    \end{subfigure}
    \begin{subfigure}{0.3\textwidth}
        \centering
        \input{Images/Prismatic}
        \caption*{$({\protect\sqmo}, 2)$-prismatic}
        \label{fig: square prismatic}
    \end{subfigure}
\caption{There are sixteen distinct colorings of the square tetromino by a set of two colors. Each of the two large polyominoes here include exactly one instance of each of these colorings, and are therefore $({\protect\sqmo},2)$-de Bruijn. The polyomino on the right is also $({\protect\sqmo},2)$-dense, and is therefore $({\protect\sqmo},2)$-prismatic.}
    \label{fig:sixteen squares}
\end{figure}

This generalizes the notion of a \bfb{de Bruijn sequence},  which is a sequence over an alphabet $A$ that includes exactly one instance of each possible subsequence of a certain length; de Bruijn sequences can be encoded as straight de Bruijn polyominoes, as in \Cref{fig:sequence}.

\begin{figure}
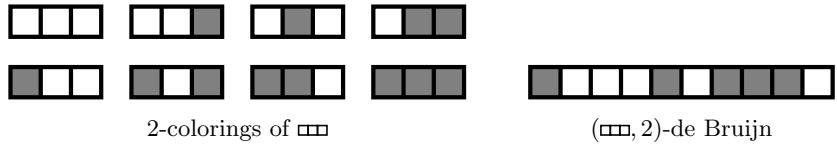

    \centering
        \begin{subfigure}{0.55\textwidth}
        \centering
        \input{Images/Colorings2}
        \caption*{2-colorings of $\stmo$}
    \end{subfigure}
    \begin{subfigure}{0.40\textwidth}
        \centering
    \input{Images/DeBruijnSequence}
        \caption*{$(\stmo, 2)$-de Bruijn}
    \end{subfigure}
\caption{The sequence (1,0,0,0,1,0,1,1,1,0) is de Bruijn, because it includes each sequence of length 3 over the alphabet $\{0,1\}$ exactly once. We can encode this sequence as the straight $10$-omino above, shading the $i$th cell if the $i$th term in the sequence is a 1. The polyomino is de Bruijn because it includes exactly one instance of each way of coloring a straight tromino by these two colors. 
De Bruijn sequences are naturally cyclic, so the sequence in this example would often be written without the last two bits, which necessarily equal the first two bits. We must write the sequence out fully in order to encode it as a polyomino.} 
    \label{fig:sequence}
\end{figure}

De Bruijn polyominoes were introduced by Condon, Wang, and Yang \cite{condon2024}. 
For $N$ any positive integer, if $P$ is a polyomino of minimum size among those polyominoes containing at least $N$ instances (translated copies) of some polyomino $p$, we say that $P$ is \bfb{$(p,N)$-dense}; density is independent of coloring. 
If $P$ is a $(p,n)$-de Bruijn polyomino that is also $(p,n^{|p|})$-dense, we say that $P$ is \bfb{$(p,n)$-prismatic.}
Condon et al. searched for $(p,n)$-prismatic polyominoes by first identifying dense polyominoes, and then constructing de Bruijn colorings.

We let $a_{p,N}$ denote the size of a $(p,N)$-dense polyomino, and we call $(a_{p,N})_{N=1}^\infty$ the \bfb{instance sequence} for $p$. 
This paper mainly develops machinery for computing instance sequences and constructing dense polyominoes. 
\Cref{sec:background} contains background material. In \Cref{sec:partialorder}, we use polyomino transformations to show that some instance sequences are upper bounds on others. In \Cref{sec:easypolyominoes,sec:otherpentominoes} we give parameterized descriptions of the instance sequences for most polyominoes with up to five cells, and we give closed forms of these sequences in \Cref{sec:closedforms}.
In \Cref{sec:instancesequences}, we show that instance sequences for polyominoes are always asymptotically linear. Finally we present several novel prismatic polyominoes, as well as other colored polyforms with de Bruijn-like properties, in \Cref{sec:novelprismatic,sec:polyforms}.

\section{Background}\label{sec:background}

\subsection{Polyominoes}
The term \bfb{polyomino} was coined by Solomon W. Golomb, who introduced and compiled many entertaining polyomino tiling problems in his book \textit{Polyominoes. Puzzles, Patterns, Problems, and Packings} \cite{Golomb}. Golomb usually considers polyominoes that differ by rotations or reflections to be the same shape, as do many works dealing with polyominoes; these equivalence classes are sometimes called \bfb{free polyominoes}. 

In this paper, we deal with \bfb{fixed polyominoes}, meaning we consider two polyominoes to be the same shape if they differ by translation only; we call these two \bfb{instances} of that shape. For example, we distinguish between these two orientations of the T tetromino, and between these two colorings of the square tetromino.
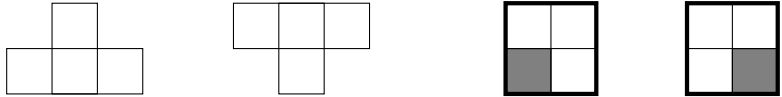
\begin{center}
\begin{tikzpicture}[scale=.6]
    \draw[-] (0,0) rectangle ++ (3,1);
    \draw[-] (1,0) rectangle ++ (1,2);
    
    \draw[-] (5,1) rectangle ++ (3,1);
    \draw[-] (6,0) rectangle ++ (1,2);

    \draw[fill=black!50] (11,0) rectangle ++ (1,1);
    \draw[ultra thick] (11,0) rectangle ++ (2,2);
    \draw[-] (11,1) --++ (2,0);
    \draw[-] (12,0) --++ (0,2);

    \draw[fill=black!50] (16,0) rectangle ++ (1,1);
    \draw[ultra thick] (15,0) rectangle ++ (2,2);
    \draw[-] (15,1) --++ (2,0);
    \draw[-] (16,0) --++ (0,2);
\end{tikzpicture}
\end{center}

We regard the cells of all polyominoes as orthogonal unit squares on the Cartesian plane, with their lower left corners having integer coordinates. We consider these the coordinates of the cell. Sometimes, we find it useful to regard a polyomino instance as the set of the coordinates of its cells. See \Cref{fig:Cartesian Grid}.

\begin{figure}
    \centering
    \input{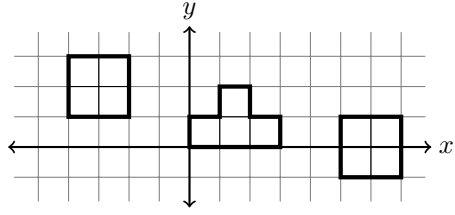}
    \caption{The instance of ${\protect \ttmo}$ has coordinates $\{(0,0), (1,0), (2,0), (1,1)\}$. The left instance of ${\protect \sqmo}$ has coordinates $\{(-4,2),(-3,2),(-4,1),(-3,1)\}$ while
    the right instance has coordinates $\{(5,0), (6,0), (5,-1), (6,-1)\}$. }
    \label{fig:Cartesian Grid}
\end{figure}

A \bfb{row} of the Cartesian plane is the set of all cells sharing some $y$-coordinate, while a \bfb{column} is the set of all cells sharing some $x$-coordinate. If a polyomino intersects a row (or column), we call that intersection a row (or column) of that polyomino.

\subsection{Triangular and Polygonal Numbers}\label{sec:polygonalnumbers}
Many of the instance sequences we will describe are related to the triangular and polygonal numbers.
The \bfb{triangular numbers} are the numbers of dots needed to make filled-in triangular arrays. The $m$th triangular number is
\begin{equation*}
    t_m \coloneq  m(m+1)/2.
\end{equation*}
The polygonal numbers generalize this idea.
 
The $m$th \bfb{$n$-gonal number}, $\poly^n_m$, is the number of dots needed to draw an array of $m$ many $n$-gons, 
$G_1$, $G_2$, ..., $G_m$,
with $G_i$ having a dot at each vertex and $i$ dots along each side, so that all the $n$-gons share one vertex and the two sides incident to that vertex. 
This implies that $G_1$ has a single dot, so $\poly^n_1 = 1$ for all $n$. The arrays in \Cref{fig:polygonalnumbers} can be used to compute the first five columns in the table below.
$$\begin{array}{l|lllll}
    m & 1 & 2 & 3 & 4 & 5  \\\hline
    t_m & 1 & 3 & 6 & 10 & 15\\
    \poly^4_m & 1 & 4 & 9 & 16 & 25\\
    \poly^5_m & 1 & 5 & 12 & 22 & 35\\
    \poly^6_m & 1 & 6 & 15 & 28 & 45
\end{array}$$

\begin{figure}
    \centering
    \input{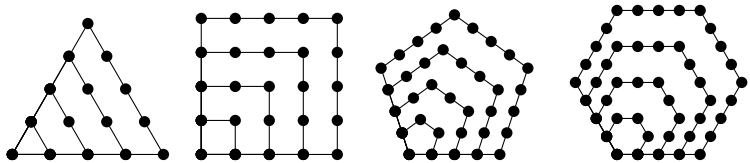}
    \caption{These arrays can be used to compute the first five triangular (3-gonal), square (4-gonal), pentagonal (5-gonal), and hexagonal (6-gonal) numbers.}
    \label{fig:polygonalnumbers}
\end{figure}

In \textit{The Book of Numbers}, John H. Conway and Richard K. Guy give a visual argument that an array of $m$ many $n$-gons, such as those in \Cref{fig:polygonalnumbers}, 
can be constructed by gluing together $n-2$ arrays of $m-1$ triangles, and a single edge with $m$ dots \cite{Conway}.
This gives the following result.

\begin{lemma}\label{lem:Conway}
    For any positive integer $m$, and any integer $n \geq 3$,
    \begin{equation*}
        \poly^n_m = m + (n-2) t_{m-1}.
    \end{equation*}
\end{lemma}

This yields the following technical lemmas which will be useful when describing the construction of certain dense polyominoes.

\begin{lemma}\label{lem:recursion}
    For any positive integer $m$, and any integer $n \geq 3$,
    \begin{equation*}
        \poly^n_{m+1} = \poly^n_m + (n-2)m + 1.
    \end{equation*}
\end{lemma}

\begin{proof}
     This is a straightforward calculation using \Cref{lem:Conway}.
\end{proof}

\begin{lemma}\label{lem:polydecompose}
    For any positive integer $N$ and positive integer $n \geq 3$, there exist a unique positive integer $m$ and unique integer $r$ such that $0 \leq r \leq (n-2)m$ and  $N = \poly^n_m + r.$
\end{lemma}

\begin{proof}
    This follows immediately from the fact that the $n$-gonal numbers are an increasing sequence of integers starting from $\poly^n_1=1$, and from \Cref{lem:recursion}.
\end{proof}

For further reading, we recommend \textit{Elementary Number Theory in Nine Chapters} by James J. Tattersall \cite{Tattersall}.

\subsection{De Bruijn Polyomino Results}

In \cite{condon2024} Condon et al. proved that $a_{\,\sqmot, n^2} = a_{\,\zmot, n^2} = (n+1)^2$, and gave a method for constructing $(\sqmo,n)$- and $(\zmo,n)$-prismatic polyominoes for all $n$.
They proved that $a_{\,\tttmo, n^2} = a_{\,\lmot, n^2} = (n+1)^2$, and constructed $(\ttmo,2)$- and $(\lmo,2)$-prismatic polyominoes. They also showed that for $p$ any straight polyomino with at least two cells, the unique shape of a $(p,N)$-dense polyomino is a straight polyomino with $|p|+N-1$ cells.

We will borrow a lemma that appeared in that paper. We provide a proof sketch for completeness.

\begin{lemma}[Appears as Lemma 2.1 in \cite{condon2024}]\label{lem:connected}
    For $N$ any positive integer and $p$ any polyomino with at least two cells,
    if $S$ is a set of cells of minimum size that contains at least $N$ instances of $p$
    then $S$ is connected.
\end{lemma}

\begin{proof}[Proof Sketch]
    It suffices to show that for $S$ any disconnected set of cells, there exists a smaller set of cells having at least as many instances of $p$. Without a loss of generality, we assume that $p$ has at least two columns. We construct this set by translating the connected components of $S$ so that two of them overlap in a single column. This decreases the total number of cells. Since $p$ has multiple columns, none of the instances of $p$ from either component translate to the same position in the new arrangement, which must then have at least as many instances of $p$ as $S$.
\end{proof}

\section{A Partial Ordering of Polyominoes}\label{sec:partialorder}

In this section, we discuss transformations of polyominoes (and other sets of cells) that for some pairs of polyominoes $p$ and $q$ show that $a_{p,N} \geq a_{q,N}$ for every positive integer $N$. 
This allows upper and lower bounds for multiple instance sequences to follow from a single argument.

\subsection{The Row Shift Transformation}

We let $\rho$ denote the transformation that shifts each row of a set of cells to the right by one cell relative to the row immediately above it, as depicted in \Cref{fig: shifts}. We call this the \bfb{row shift} transformation. We note that the transformation acts cell-wise and is invertible.

\begin{figure}
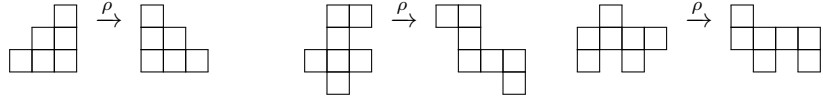

    \centering
    \begin{subfigure}{0.3\textwidth}
        \centering        \input{Images/shiftTransformation3}
    \end{subfigure}
    \begin{subfigure}{0.3\textwidth}
        \centering        \input{Images/shiftTransformation2}
    \end{subfigure}
    \begin{subfigure}{0.3\textwidth}
        \centering        \input{Images/shiftTransformation4}
    \end{subfigure}
    \caption{The row shift transformation.}
    \label{fig: shifts}
\end{figure}

Although the row shift sometimes transforms a polyomino into a disconnected set of cells, it often maps dense polyominoes to dense polyominoes, as described by the following lemma. The idea is implicit in the work of Condon et al. \cite{condon2024}. 

\begin{lemma}\label{lem:shiftorder}
    For any positive integer $N$ and polyomino $p$ such that $q = \rho(p)$ is a polyomino and $|p| > 1$, if $P$ is a $(p,N)$-dense polyomino then $Q = \rho(P)$ is a $(q,N)$-dense polyomino. Furthermore, $a_{p,N} = a_{q,N}$. 
\end{lemma}

\begin{proof}
    For any set of cells $S$, each instance of $p$ in $S$ is transformed cell-wise into an instance of $q$ in $\rho(S)$. Symmetrically, any instance of $q$ is transformed into an instance of $p$ in $\rho^{-1}(S)$.

    Thus $Q$ is a set of cells of size $a_{p,N}$ containing at least $N$ instances of $q$. By \Cref{lem:connected}, $a_{q,N} \leq a_{p,N}$. By symmetry, $a_{p,N} \leq a_{q,N}$. Therefore $a_{p,N} = a_{q,N}$ and $Q$ is a $(q,N)$-dense polyomino by \Cref{lem:connected}.
\end{proof}

\subsection{The Column Slide Transformation}
We say a polyomino is \bfb{bottom-aligned} if all of its columns are convex (having no gaps) and the bottom-most cells in all of its columns belong to the same row.

We let $\sigma$ denote the transformation that maps a polyomino to the bottom-aligned polyomino with the same number of columns and cells in each column, as depicted in \Cref{fig: slides}. We call this the \bfb{column slide} transformation. 
Unlike the row shift, the column slide always maps polyominoes to polyominoes. 

\begin{figure}
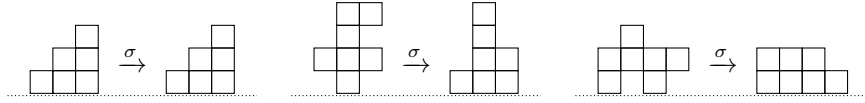

    \centering
    \begin{subfigure}{0.3\textwidth}
        \centering        \input{Images/slideTransformation3}
    \end{subfigure}
    \begin{subfigure}{0.3\textwidth}
        \centering        \input{Images/slideTransformation2}
    \end{subfigure}
    \begin{subfigure}{0.35\textwidth}
        \centering        \input{Images/slideTransformation4}
    \end{subfigure}
    \caption{The column slide transformation. It is helpful to imagine the cells of each column sliding down until resting upon the dotted line.}
    \label{fig: slides}
\end{figure}

We will show that if $P$ contains $N$ instances of a polyomino $p$, then $\sigma(P)$ contains at least $N$ instances of $\sigma(p)$, which shows that $a_{p,N} \geq a_{\sigma(p), N}$ for every polyomino $p$ and positive integer $N$.

\begin{lemma}\label{lem: slide instances}
    For any polyominoes $P$ and $p$, the number of instances of $\sigma(p)$ in $\sigma(P)$ is at least the number of instances of $p$ in $P.$ 
\end{lemma}

\begin{proof}
    For any polyominoes $X$ and $Y$, let $I(X,Y)$ denote the number of instances of $Y$ in $X$. For $i \leq j$,  $X_{[i,j]}$ denote the polyomino formed by the $i$th through $j$th columns of $X$, indexed from left to right starting from 1.
    
    Let $w$ be the number of columns of $p$ and let $W$ be the number of columns of $P$. Since each instance of $p$ in $P$ spans exactly $w$ consecutive columns, these instances can be grouped by which columns they occupy,
    $$I(P,p) = \sum_{k=1}^{W-w+1} I(P_{[k, k+w-1]},p).$$

    Let $c_i$ and $C_i$ denote the number of cells in the $i$th columns of $p$ and $P$ respectively, indexed from left to right. Note that every $P_{[k, k+w-1]}$ contains at most $C_{k-1+i} - c_i + 1$ instances of $p$ for any $i = 1, ..., w,$ since this is the maximum number of distinct positions in which the $i$th column of $p$ can fit within the $i$th column of $P_{[k, k+w-1]}$; this can be counted by considering the valid positions of the bottom cell of the $i$th column of $p$, which must be within the $i$th column of $P_{[k, k+1-1]}$ but not among its top $c_{i}-1$ cells. Therefore
    $$I(P_{[k,k+w-1]},p) \leq \min_{i=1,...,w}(C_{k+i-1} -c_i + 1) \quad \text{for } k = 1,...,W-w+1,$$
    so
    $$I(P,p) = \sum_{k=1}^{W-w+1}I(P_{[k,k+w-1]},p) \leq \sum_{k=1}^{W-w+1}\min_{i=1,...,w}(C_{k+i-1} -c_i + 1).$$

    Now let $q = \sigma(p)$ and $Q = \sigma(P)$. Because $q$ and $Q$ are bottom-aligned, the number of instances of $q$ within $Q_{[k,k+w-1]}$ is exactly $\min_{i=1,...,w}(C_{k+i-1} -c_i + 1)$, these being the translations of $q$ from its lowest possible position in $Q_{[k,k+w-1]}$ to its highest possible position. It follows that
    $$I(P,p) \leq \sum_{k=1}^{W-w+1}\min_{i=1,...,w}(C_{k+i-1} -c_i + 1) = \sum_{k=1}^{W-w+1}I(Q_{[k,k+w-1]},q) = I(Q,q).$$
    Therefore $I(Q,q)$, which is the number of instances of $\sigma(p)$ in $\sigma(P)$, is at least $I(P,p)$, which is the number of instances of $p$ in $P$.
\end{proof}

\begin{lemma}\label{thm:slideorder}
    For any polyomino $p$ and positive integer $N$, $a_{p,N} \geq a_{\sigma(p), N}$.
\end{lemma}

\begin{proof}
    Let $P$ be a $(p,N)$-dense polyomino. By \Cref{lem: slide instances}, $\sigma(P)$ contains at least $N$ instances of $\sigma(p)$, so $\abs{\sigma(P)} \geq a_{\sigma(p),N}$. Since the column slide preserves the number of cells in the polyomino, 
    $a_{p,N} = \abs{P} =  \abs{\sigma(P)} \geq a_{\sigma(p), N}.$
\end{proof}

We say a polyomino is \bfb{left-aligned} if all of its rows are convex and the left-most cells in all of its rows belong to the same column. 
We say a polyomino is \bfb{corner-aligned} if it is both left-aligned and bottom-aligned.

\begin{lemma}\label{lem:corner-aligned}
    If $P$ is a left-aligned polyomino, then $\sigma(P)$ is corner-aligned.
\end{lemma}

\begin{proof}
    Since $\sigma(P)$ is bottom-aligned by definition, it remains to show $\sigma(P)$ is left-aligned. 
    Because $P$ is left-aligned, the number of cells in the $i$th column of $P$ (indexing from left to right) is the same as the number of rows having at least $i$ cells. Therefore the column sizes of $\sigma(P)$ are nonincreasing from left to right. Because $\sigma(P)$ is bottom-aligned, this guarantees each row has a cell in the left-most column, and that there are no gaps in rows, so $\sigma(P)$ is left-aligned.
\end{proof}

We can similarly define a row slide function $\sigma'$ that maps a polyomino to the left-aligned polyomino with the same number of rows and cells in each row. By symmetry, all of the previous results in this section have analogs which apply to $\sigma'$. 

\begin{theorem}\label{thm:corner-aligned}
    For any positive integer $N$ and any polyomino $p$ that is bottom-, left-, or corner-aligned, there exists a $(p,N)$-dense polyomino that is also bottom-, left-, or corner-aligned respectively.
\end{theorem}

\begin{proof}
     Let $P$ be any $(p,N)$-dense polyomino. In the case that $p$ is bottom-aligned, note that $\sigma(P)$ contains at least as many instances of $\sigma(p)=p$ as $P$ by \Cref{lem: slide instances}. The density of $\sigma(P)$ then follows from the density of $P$.
     Similarly, if $p$ is left-aligned then $\sigma'(P)$ is $(p,N)$-dense as well as left-aligned.

    In the case that $p$ is corner-aligned, $\sigma'(P)$, and $\sigma(\sigma'(P))$ are $(p,N)$-dense by the preceding argument. Since $\sigma'(P)$ is left-aligned, $\sigma(\sigma'(P))$ is corner-aligned by \Cref{lem:corner-aligned}.
\end{proof}

\subsection{Relating Instance Sequences}\label{sec:relatinginstancesequences}

Using multiple row shifts and column slides, we can relate the instance sequences for more polyominoes. For example, 
\begin{samepage}
$$\sqmo \xleftrightarrow{\text{row shift}} \zmo \xrightarrow{\text{column slide}} \ttmo \xleftrightarrow{\text{row shift}} \lmo$$
and so 
\begin{equation}\label{eq:tetrominoranking}
    a_{\,\sqmot,N} = a_{\,\zmot,N} \geq a_{\,\tttmo,N} = a_{\,\lmot,N}
\end{equation}\end{samepage}

\noindent
for every positive integer $N$ by \Cref{lem:shiftorder} and \Cref{thm:slideorder}. Analogous information for pentominoes is displayed in \Cref{fig:posets}. 

\begin{figure}
    \centering
        \input{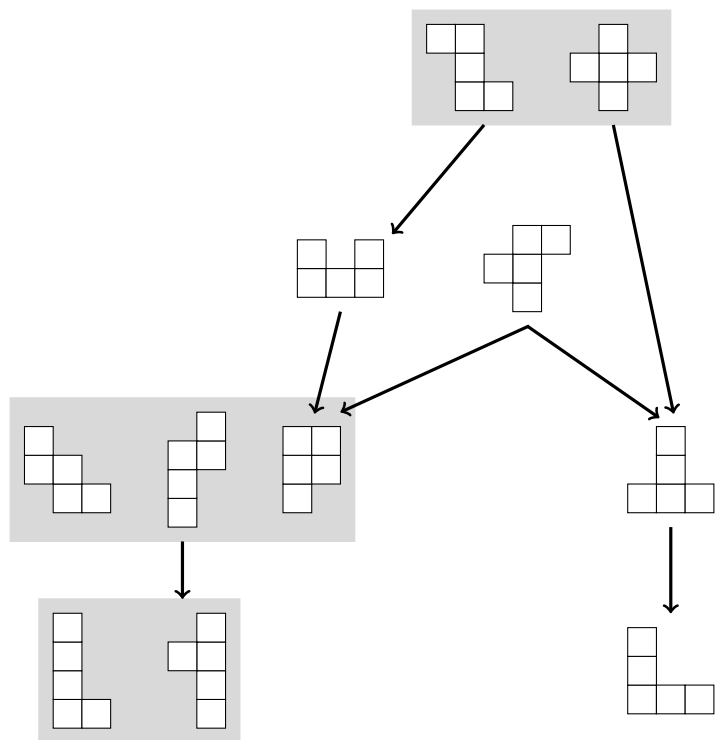}
        
    \caption{Polyominoes grouped in a box are related by row shifts and rotations, and therefore have the same instance sequence. An arrow indicates that one group of polyominoes maps to another under a column slide (allowing rotations).}
    \label{fig:posets}
\end{figure}

\section{The L Tromino, Tetrominoes, and L, N, P, W, and Y Pentominoes}\label{sec:easypolyominoes}
In this section, we will construct several polyominoes that give upper bounds for $a_{\,\tmot,N}$, $a_{\,\sqmot,N}$, and $a_{\,\tppmo,N}$ respectively. We then show that these are equal to lower bounds for $a_{\,\tmot,N}$, $a_{\,\lmot,N}$, and $a_{\,\tlpmo,N}$ respectively. Combining these bounds with the result from \Cref{sec:relatinginstancesequences} that
$$a_{\,\sqmot,N} = a_{\,\zmot,N} \geq a_{\,\tttmo,N} = a_{\,\lmot,N}$$
and the analogous information for pentominoes from \Cref{fig:posets}, we find instance sequences for $\tmo$, $\sqmo$, $\zmo$, $\ttmo$, $\lmo$, $\lpmo$,$\smo$, $\ppmo$, $\wmo$, and $\ypmo$.

\subsection{Polyomino Addition}

There is often a striking resemblance between a polyomino $p$ and certain $(p,N)$-dense polyominoes, as depicted in \Cref{fig: self-similarity}. 
Each dense polyomino in the figure can be constructed by a process we call \bfb{polyomino addition}.

\begin{figure}
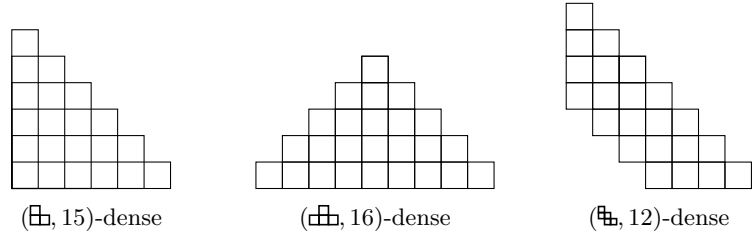

    \centering
    \begin{subfigure}{0.3\textwidth}
        \centering
        \input{Images/petiteTrominoIsland}
        \caption*{$(\tmo, 15)$-dense}
    \end{subfigure}
    \begin{subfigure}{0.3\textwidth}
        \centering
        \input{Images/petiteTetrominoIsland}
        \caption*{$(\ttmo, 16)$-dense}
    \end{subfigure}
    \begin{subfigure}{0.3\textwidth}
        \centering
        \input{Images/petitePentominoIsland}
        \caption*{$(\wmo, 12)$-dense}
    \end{subfigure}
    \caption{Some $(p,N)$-dense polyominoes which resemble $p$. 
    We note that $15$ is the fifth triangular number, $16$ is the fourth square number, and $12$ is the third pentagonal number.} 
    \label{fig: self-similarity}
\end{figure}

Geometrically, the sum of two polyomino shapes $p + q$ can be found by taking the union of all those instances of~$q$ for which some fixed cell in~$q$ aligns with a cell in some fixed instance of~$p$. For example, $\tmo + \dmo = \ppmo$, as shown in \Cref{fig: polyomino sum}.

\begin{figure}
    \centering
    \input{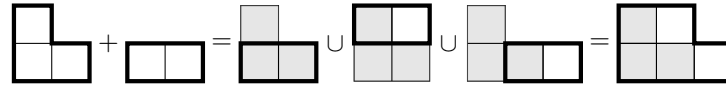}
    \caption{We find the shape of ${\protect \tmo}+{\protect\dmo}$ by taking the union of each instance of ${\protect \dmo}$ such that its left cell aligns with a cell in a fixed instance of ${\protect\tmo}$. We shade the cells of the fixed instance of ${\protect \tmo}$ for reference.}
    \label{fig: polyomino sum}
\end{figure}

An alternative perspective is to associate each polyomino with the set of Cartesian coordinates of its cells. The polyomino sum is then the sumset of these coordinates as vectors. 
For example, $\{(0,0), (1,0), (0,1)$\} is an instance of $\tmo$ and $\{(0,0), (1,0)\}$ is an instance of $\dmo$, and their polyomino sum 
$$\{(0,0), (1,0), (0,1)\} + \{(0,0), (1,0)\} = \{(0,0), (1,0), (0,1), (2,0), (1,1)\}$$
is an instance of $\ppmo$. With this approach it is clear polyomino addition is associative and commutative.

Borrowing notation from additive combinatorics, for any polyomino $p$ and positive integer $m$ we define $mp$ to be the sum of $p$ with itself $m$ times,
$$m p:=  \underbrace{p + p + ... + p}_\text{\text{$m$ terms}}.$$ 
We call $mp$ a \bfb{multiple} of $p$. Each $(p,N)$-dense polyomino in \Cref{fig: self-similarity} is a multiple of $p$.

As we will see, $mp$ is often $(p, f(m))$-dense, where $f$ is some increasing function of $m$. This is useful in constructing the subsequence $(a_{p, f(m)})_{m=1}^\infty$ of the instance sequence for $p$. 
For $N$ an integer with $f(m) < N < f(m+1)$ for some $m$, we can often construct a $(p,N)$-dense polyomino by strategically adding cells to $m p$.
Interestingly, $f(m)$ is often simply $\poly^{|p|}_m$, the $m$th $|p|$-gonal number. Examples of this are noted in \Cref{fig: self-similarity}.

\subsection{A Construction for L Trominoes}\label{sec:constructiontrominoes}

Our construction of dense polyominoes for L trominoes involves the triangular numbers.
Consider the following multiples of $\tmo$.
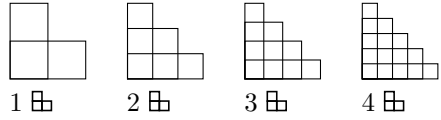
\begin{center}
\begin{tabular}{llll}
    \zig 2 & \zig 3 & \zig 4 & \zig 5\\
    $1\ \tmo$ & $2\ \tmo$ & $3\ \tmo$ & $4\ \tmo$
\end{tabular}
\end{center}
Note that the number of cells in $m \,\tmo$ is the same as the number of dots in the triangular array used to compute $t_{m+1}$ (see \Cref{fig:polygonalnumbers}).
Each cell in $m\,\tmo$ corresponds to an instance of $\tmo$ in $(m+1)\,\tmo$, so in general $m\,\tmo$ has at least $t_m$ instances of $p$ and has $t_{m+1}$ cells.
This shows that
$$a_{\,\tmot, t_m} \leq t_{m+1}.$$

\begin{figure}
    \centering
    \input{Images/StairClimb}
    \caption{From left to right, these polyominoes respectively contain $t_4=10$, 11, 12, 13, 14, and $t_5 = 15$ instances of ${\protect \tmo}$. }
    \label{fig:stairclimb}
\end{figure}

We can generalize this upper bound by adding cells at the right end of the bottom rows of $m\,\tmo$ in order to efficiently create more instances of $\tmo$, as in \Cref{fig:stairclimb}.
The general upper bound is described by the following lemma.

\begin{lemma}\label{thm:tmo_upper_bound}
    For any positive integer $m$ and integer $r$ such that $0 \leq r \leq m$,
    $$a_{\,\tmot, t_{m}+r} \leq t_{m+1} + r+ \ceil{ \frac rm}.$$
\end{lemma}

\begin{proof}
    We will construct a polyomino $P$ with $t_m+r+\ceil{r/m}$ cells having at least $t_m+r$ instances of $\tmo$.

    If $r=0$, we simply let $P = m \,\tmo$, which has $t_{m+1}$ cells and at least $t_m$ instances of $\tmo$.

    If $r > 0$, we adjoin a cell to the right end of each of the bottom $r+1$ rows of $m\, \tmo$.
    This creates $r$ new instances of $\tmo$ while adding $r+1$ cells, for a total of at least $t_m+r$ instances and $t_{m+1} + r + 1$ cells. Since $1 \leq r \leq m$, we have $\ceil{r/m} = 1$ so the number of cells is $t_{m+1} + r+ \ceil{r/m}$.
\end{proof}

We will prove this construction is dense in \Cref{sec:mainresults}.

\subsection{A Construction for P polyominoes}

We define a \bfb{P $n$-omino}, for $n \geq 4$, to be a left-aligned polyomino with two rows that has $2$ cells in the top row and $n-2$ cells in the bottom row. This generalizes the naming convention of the P pentomino, as depicted in \Cref{fig: P s-ominos}.

\begin{figure}
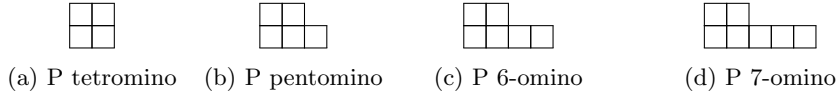

    \centering
    \begin{subfigure}{0.2\textwidth}
        \centering
        \begin{ytableau}
            \ & \ \\
            \ \ &\ \\
        \end{ytableau} 
        \caption{P tetromino}
    \end{subfigure}
    \begin{subfigure}{0.22\textwidth}
        \centering
        \begin{ytableau}
            \ & \ \\
            \ \ &\ & \  \\
        \end{ytableau} 
        \caption{P pentomino}
    \end{subfigure}
    \begin{subfigure}{0.23\textwidth}
        \centering
        \begin{ytableau}
            \ & \ \\
            \ & \ & \  & \ \\
        \end{ytableau} 
        \caption{P 6-omino}
    \end{subfigure}
    \begin{subfigure}{0.3\textwidth}
        \centering
        \begin{ytableau}
            \ &\  \\
            \ & \ & \ & \ &\ \\
        \end{ytableau} 
        \caption{P 7-omino}
        \label{fig: P hexomino}
    \end{subfigure}
    \caption{The P $n$-ominoes for $n = 4$, 5, 6,  and 7.}
    \label{fig: P s-ominos}
\end{figure}

\begin{figure}
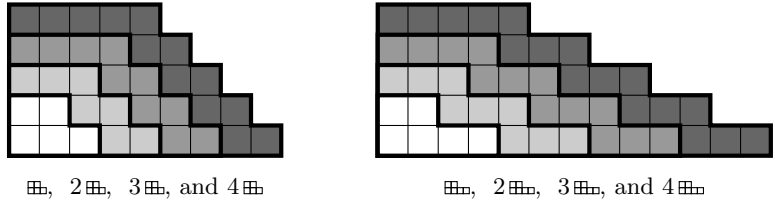

    \centering
    \begin{subfigure}{0.375\textwidth}
        \centering
        \input{Images/PpentominoMultiples}
        \caption*{$\tppmo$, \ $2 \, \tppmo$, \ $3 \, \tppmo$, and $4 \, \tppmo$ }
    \end{subfigure}
    \begin{subfigure}{0.542\textwidth}
        \centering
        \input{Images/PhexominoMultiples}
        \caption*{$\tphmo$, \ $2 \, \tphmo$, \ $3 \, \tphmo$, and $4 \, \tphmo$ }
    \end{subfigure}
    \caption{Multiples of P $n$-ominoes superimposed.} 
    \label{fig: Pnomino multiples}
\end{figure}

\Cref{fig: Pnomino multiples} shows some multiples of P $n$-ominoes. From these examples, we notice a consistent shape. We define an \bfb{$(a,b)$-staircase} to be a left-aligned polyomino having $a$ rows, with $a$ cells in the top row, with each other row having $b$ more cells than the row above it. 

We can express the number of cells in the $(a,b)$-staircase in terms of polygonal numbers, using \Cref{lem:Conway}.
We note that the $(a,0)$-staircase is just an $a \times a$ square, having $a^2 = \poly^4_a = a + 2t_{a-1}$ cells in total. The $(a,b)$-staircase has $b t_{a-1}$ cells to the right of its $a \times a$ square, for a total of $a + (b+2)t_{a-1}$ cells.  

\begin{lemma}\label{lem:staircase}
    For any P $n$-omino $p$ and positive integer $m$, the polyomino $m p$ is an $(m+1,n-4)$-staircase and has $\poly^n_{m+1}$ cells.
\end{lemma}

\begin{proof}
    We proceed by induction on $m$. 
    For our base case, $m=1$, note that $p$ itself is a $(2,n-4)$-staircase. We hypothesize that for some integer $m>1$, $(m-1)p$ is an $(m,n-4)$ staircase.
    
    We compute $mp$ by taking the union of all instances of $p$ with a bottom left cell in some fixed instance of $(m-1)p$. This union includes $(m-1)p$, and has $(n-3)$ additional cells on the right of each row; these additional cells belong to the instance of $p$ with its bottom left cell at the end of that row of $(m-1)p$.

    The union also includes a new row above $(m-1)p$, with its left-most cell in the same column as the other rows, having one more cell than the top row of $(m-1)p$. The cells in the new row belong to instances of $p$ with their bottom left-most cell in the top row of $(m-1)p$.

    Therefore $mp$ is left-aligned, and has $m+1$ rows and $m+1$ cells in its top row. The row beneath the top row has $m+(n-3)$ cells, for a difference of $n-4$ cells. Each other row has $(n-4)$ fewer cells than the row above it, as this was the case for $(m-1)p$ and both rows increased in length by the same amount.

    Therefore $mp$ is a $(m+1,n-4)$-staircase. It follows immediately that the number of cells in $mp$ is
    $$m+1 + (n-2) t_m$$
    which is the $m+1$st $n$-gonal number $\poly^n_{m+1}$ by \Cref{lem:Conway}.
\end{proof}

For any P $n$-omino $p$, each cell in $m p$ corresponds to an instance of $p$ in $(m+1)p$. Then, by \Cref{lem:staircase}, $m p$ has at least $\poly^n_m$ instances of $p$ and has $\poly^n_{m+1}$ cells. This shows that
$$a_{p, \poly^n_m} \leq \poly^n_{m+1}.$$

\begin{figure}
    \centering
    \input{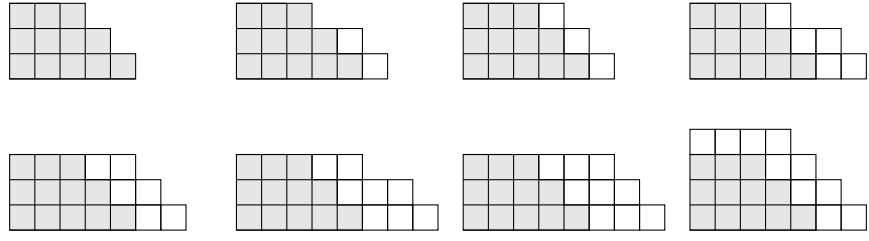}
    \caption{From left to right, top to bottom, these polyominoes respectively contain $\poly^5_2 = 5$, 6, 7, 8, 9, 10, 11, and $\poly^5_3=12$ instances of ${\protect\ppmo}$.}
    \label{fig:StairClimbPentomino}
\end{figure}

We can generalize this upper bound by adding cells to the ends of the rows of $mp$ to create more instances of $p$, adding a single cell to the right end of each row, working from the bottom up. We repeat this process until reaching the next $n$-gonal number of cells, as in \Cref{fig:StairClimbPentomino}.
The general upper bound is described by the following lemma. 

\begin{lemma}\label{lem:Pnominoupperbound}
    For any positive integer $m$ and integer $r$ such that $0 \leq r \leq (n-2)m$,
    $$a_{p,\poly^n_{m} + r } \leq \poly^n_{m+1} + r + \ceil{\frac rm}$$
    where $p$ denotes the P $n$-omino.
\end{lemma}
\begin{proof}
    We will construct a polyomino $Q$ with $\poly^n_m + r + \ceil{r/m}$ cells having at least $\poly^n_{m} + r$ instances of $p$.
    We note that $mp$ is an $(m+1,n-4)$-staircase, which has at least $\poly^n_m$ instances of $p$ and exactly $\poly^n_{m+1}$ cells.

    Applying the quotient-remainder theorem, let $t$ and $s$ be integers such that $r = mt+s$ and $0 \leq s < m$. Note that $t = \floor{r/m}$, and $t=r/m$ iff $s = 0$.
    
    We construct $Q$ from $mp$ in $t+1$ steps: first, we add one cell to the right end of each row of $mp$; we repeat this a total of $t$ times; lastly, if $s > 0$ we add one cell to the right end of the bottom $s+1$ rows of the polyomino.

    In each of the first $t$ steps we create $m$ new instances of $p$, while adding $m+1$ cells. In the final step we create $s$ new instances of $p$, while adding $s+1$ cells if $s > 0$.
    The number of instances of $p$ in $Q$ is therefore at least $\poly^n_m + mt + s = \poly^n_m + r$, as was to be shown.

    In the case that $s=0$, the number of cells added to the staircase is $$t(m+1) = tm+t = r + t.$$ In this case $t = r/m$. Because $t$ is an integer, $t = \ceil{r/m}$ and total number of cells in $Q$ is $\poly^n_{m+1} + r + \ceil{r/m}$
    as was to be shown.
    
    In the case that $s>0$, the number of cells added to the staircase is $$t(m+1) + s+1= mt+s + t +1 = r + \floor{\frac rm} + 1.$$
    Since $r/m$ is not an integer in this case, $\floor{r/m}+1 = \ceil{r/m}$. The total number of cells in $Q$ is therefore $\poly^n_{m+1} + r + \ceil{r/m}$
    as was to be shown.
\end{proof}
We will prove this construction is dense in \Cref{sec:mainresults}. We remark that this construction can be altered in the case $(n-3) m < r \leq (n-2)m$, so that the final cells added form a new top row of the shape rather than extending existing rows. 
This has the aesthetic benefit that the constructed shape contains $mp$, but fits within $(m+1)p$.

\subsection{A Lower Bound for L Polyominoes}\label{sec:lowerbounds} \label{sec:lowerboundforLominoes}

We define an \bfb{L $n$-omino}, for $n \geq 3$, to be a left-aligned polyomino with two rows that has 1 cell in the top row and $n-1$ cells in the bottom row. This generalizes the naming convention of the L tromino, L tetromino, and L pentomino, as depicted in \Cref{fig: L polyominoes}. We treat $n$ as fixed throughout the section.

\begin{figure}
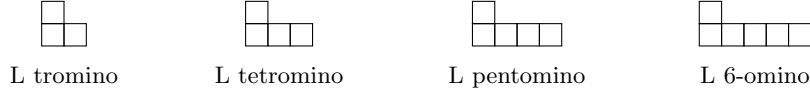

    \centering
    \begin{subfigure}{0.2\textwidth}
        \centering
        \begin{ytableau}
            \ & \none  \\
            \ & \  \\
        \end{ytableau} 
        \caption*{L tromino}
        \label{fig: L tromino}
    \end{subfigure}
    \begin{subfigure}{0.25\textwidth}
        \centering
        \begin{ytableau}
            \ & \none & \none \\
            \ & \ & \ \\
        \end{ytableau} 
        \caption*{L tetromino}
        \label{fig: L tetromino}
    \end{subfigure}
    \begin{subfigure}{0.25\textwidth}
        \centering
        \begin{ytableau}
            \ & \none & \none & \none \\
            \ & \ & \  & \ \\
        \end{ytableau} 
        \caption*{L pentomino}
        \label{fig: L pentomino}
    \end{subfigure}
    \begin{subfigure}{0.25\textwidth}
        \centering
        \begin{ytableau}
            \ & \none & \none & \none \\
            \ & \ & \ & \  & \ \\
        \end{ytableau} 
        \caption*{L 6-omino}
        \label{fig: L hexomino}
    \end{subfigure}
    \caption{The L $n$-ominoes for $n=3$, 4, 5, and 6.}
    \label{fig: L polyominoes}
\end{figure}

The goal of this section is to give a lower bound on $a_{p,N}$ when $p$ is an L $n$-omino. 
Within this section, we assume the coordinates of all cells are nonnegative, so that all cells rest in the first quadrant of the Cartesian plane.

Within this section, we define the \bfb{diagonal} $D_i$ to be the set of cells in the first quadrant with coordinates along the line $y = - \frac{1}{n-2}(x-i)$.  It is straightforward to check that the cell with coordinates $(a,b)$ belongs to the diagonal $D_{(n-2)b+a}$, so that $D_0, D_1, D_2, ...$ partition the cells in the first quadrant.
We call the intersection of a diagonal with a polyomino a diagonal of that polyomino. We start with some technical lemmas about diagonals.

\begin{lemma}\label{lem:diagonalbonus}
    For any L $n$-omino $p$ and positive integer $N$, if $P$ is a $(p,N)$-dense polyomino with $k$ diagonals, then $|P| \geq N+k$.
\end{lemma}

\begin{proof}
    Notice that $P$ has $N$ cells which belong to the top row of an instance of $p$. None of these can be the right-most cell in any diagonal of $P$, as the top cell and right-most cell of any instance of $p$ belong to the same diagonal.
    Therefore the top cells of instances of $p$, together with the right-most cells of each diagonal of $P$, form $N+k$ distinct cells in $P$.
\end{proof}

For $p$ an L $n$-omino, \Cref{thm:corner-aligned} implies there exist $(p, N)$-dense polyominoes which are corner-aligned, and we will consider instances of these which contain the cell with coordinates $(0,0)$. As shorthand, corner-aligned polyominoes in the first quadrant that include the cell with coordinates $(0,0)$ are called \bfb{quadrant-aligned}.

\begin{lemma}\label{lem:diagorder_first}
    For any positive integer $N$, let $P$ be a quadrant-aligned $(p,N)$-dense polyomino where $p$ is an L $n$-omino. If $P$ intersects a diagonal, then $P$ intersects every diagonal of lower index.
\end{lemma}

\begin{proof}
    We proceed by induction on the index $k$ of the diagonals that intersect $P$. The base case $k=0$ holds vacuously.

    For our inductive hypothesis, suppose for some positive integer $k$ that if $P$ intersects a diagonal of index $k$ then it intersects every diagonal of lower index. For our inductive step, suppose $P$ intersects the diagonal of index $k+1$ at a cell with coordinates $(x,y)$. Since $k$ is positive, $x > 0$ or $y>0$.

    Suppose $x > 0$. Then $P$ also contains the cell $(x-1,y)$ since $P$ is corner-aligned. This cell belongs to the diagonal of index $k$, so the statement follows by induction.

    Now suppose $x=0$ and $y > 0$. Since $P$ is $(p,N)$-dense, the cell $(0,y)$ belongs to some instance of $p$ in $P$, and must be one of the two left-most cells of that instance. We will show that in either case $P$ also contains the cell $(n-2,y-1)$; this cell belongs to the diagonal of index $k+1$ so that the statement follows from the previous argument with $x>0$.
 
    If $(0,y)$ is the top left cell of an instance of $p$, then $(n-2,y-1)$ is the right-most cell of that instance of $p$ and belongs to $P$.
    If $(0,y)$ is the bottom left cell of an instance of $p$, then $(n-2,y)$ is the right-most cell of that instance of $p$ and belongs to $P$; then $P$ also contains $(n-2,y-1)$ since $P$ is corner-aligned. So $P$ contains the cell $(n-2,y-1)$ in either case, as was to be shown.
\end{proof}

\begin{lemma}\label{lem:diagonalscontain}
    For any positive integer $M$, there exist a unique integer $m$ and a unique integer $u$ such that $M = (n-2) m + u$ and $0 \leq u < n-2$. Furthermore, the union of the first $M+1$ diagonals $\bigcup_{i=0}^M D_i$ forms a polyomino having at most $\poly^n_m+mu$ instances of the L $n$-omino $p$.
\end{lemma}

\begin{proof} 
    The existence and uniqueness of $m$ and $u$ follow from the quotient-remainder theorem. We will count how many instances of $p$ are created as we add each diagonal to the union. It may be helpful to refer to \Cref{fig:diagonal}.

    For some positive integer $k$, let $Q$ be the union $\bigcup_{i=0}^{k-1} D_i$.
    Note that the diagonal $D_k$ has $\floor{k/(n-2)}+1$ cells. 
    When we add $D_k$ to the union, any new instance of $p$ can only intersect $D_k$ at its top- or right-most cell, as no cell in $D_k$ is immediately left from a cell in $Q \cup D_k$. 
    The total number of instances of $p$ in the first quadrant which have a top cell or right-most cell in $D_k$ is
    one less than the number of cells in $D_k$, or $\floor{k/(n-2)}$. Therefore $D_k$ adds at most $\floor{k/(n-2)}$ instances of $p$ to the shape. 

        \begin{figure}
        \centering

        \begin{tikzpicture}[scale=.5]
            \draw[thick, <->] (9.25,0) -- (0,0) -- (0,6.25);
        
            \begin{scope}
            \clip (0,0) rectangle ++ (9,6);
            \foreach \x in {0,...,18}{
            \foreach \y in {0,...,\x}{
                \numcell{\x-2*\y}{\y}{\x}
            }
            }
                
            \end{scope}

            \foreach \z in {0,1,2,3}{
                \draw[ultra thick] (6-2*\z,\z) --++ (3,0)--++(0,1)--++(-2,0)--++(0,1)--++(-1,0)--cycle;
            }
            
        \end{tikzpicture}
        
        \caption{
        Each cell in the first quadrant is labeled with the index of its diagonal, where $n=4$. 
        The $i$th diagonal $D_i$ contains $\floor{i/(n-2)}+1$ cells. Note that if $p$ is an L $n$-omino, then the top- and right-most cells of any instance of $p$ in the first quadrant are consecutive cells of some diagonal. The number of instances of $p$ that have a top- or right-most cell in a diagonal $D_i$ is then one less than the number of cells in $D_i$, as demonstrated here for $D_8$ which has five cells and intersects four instances of ${\protect \lmo}$.} 
        \label{fig:diagonal}
    \end{figure}

    Therefore adding diagonals $D_{m(n-2)}$ through $D_{(m+1)(n-2)-1}$ to the union creates at most $m$ instances of $p$ per diagonal; so adding the first $n-2$ diagonals creates no instances of $p$, the next $n-2$ diagonals each create up to 1 instance of $p$, the next $n-2$ diagonals each create up to 2 instances of $p$, and so forth.
    Then the number of instances of $p$ in $\bigcup_{i=0}^{m(n-2)} D_i$ is at most $$(n-2)t_{m-1} + m$$ which is $\gamma_n^m$ by \Cref{lem:Conway}. The number of instances of $p$ in $\bigcup_{i=0}^{m(n-2)+u} D_i$ is then at most $\poly^m_n + mu$ as was to be shown.
\end{proof}

We are now ready to give a lower bound on the instance sequence for L $n$-ominoes. This will also serve as a lower bound for many other polyominoes.

\begin{lemma}\label{lem:Lnominolowerbound}
For any positive integer $m$ and integer $r$ such that $0 \leq r \leq (n-2)m$,
    $$a_{p,\poly^n_{m} + r} \geq \poly^n_{m+1} + r + \ceil{\frac rm}$$
    where $p$ denotes the L $n$-omino.
\end{lemma}

\begin{proof}
    Consider $P$, a $(p,\poly^n_{m} + r)$-dense polyomino. We will show $|P| \geq \poly_{m+1}^n + r + \ceil{r/m}.$ 
    By \Cref{thm:corner-aligned}, we may assume without a loss of generality that $P$ is quadrant-aligned. Let $k$ be the highest index of any diagonal that $P$ intersects; by \Cref{lem:diagorder_first}, $P$ also intersects every diagonal of lower index.

    The number of instances of $p$ in $P$ is $\poly^n_m + m \cdot \frac{r}{m}$ , so \Cref{lem:diagonalscontain} implies $k \geq (n-2)m + \ceil{r/m}$. Then $P$ has at least $1+(n-2)m + \ceil{r/m}$ diagonals. By \Cref{lem:diagonalbonus} and \Cref{lem:recursion},
    \begin{align*}
        |P|  &\geq \poly^n_{m} + r + 1+(n-2)m +\ceil{\frac rm}\\
        &= \poly_{m+1}^n + r + \ceil{\frac rm}.
    \end{align*}
\end{proof}

\subsection{Combining Upper and Lower Bounds}\label{sec:mainresults}
The lower bound on instance sequences given by \Cref{lem:Lnominolowerbound} is also a lower bound for L trominoes and P polyominoes. Furthermore, it is equal to the upper bounds we have given for those instance sequences.

\begin{theorem}\label{thm:trominomain}
    For any positive integer $N$, there exist a unique positive integer $m$ and unique integer $r$ such that $0 \leq r \leq m$ and $N = t_m + r$. Furthermore, 
    $$a_{\,\tmot, N} = t_{m+1}  + r+ \ceil{\frac rm}.$$
\end{theorem}

\begin{proof}
    The existence and uniqueness of $m$ and $r$ follow from \Cref{lem:polydecompose}.
    By \Cref{lem:Lnominolowerbound} and \Cref{thm:tmo_upper_bound}, 
    $$t_{m+1} + r+ \ceil{\frac rm} \leq a_{\,\tmot, t_m+r} \leq t_{m+1}  + r+ \ceil{\frac rm},$$
    and therefore $a_{\,\tmot, N} = t_{m+1}  + r+ \ceil{r/m}.$
\end{proof}

\begin{theorem}\label{thm:main}
    For any positive integer $N$ and positive integer $n \geq 4$, there exist a unique positive integer $m$ and unique integer $r$ such that $0 \leq r \leq (n-2)m$ and $N = \poly^n_m + r$. Furthermore, if $p$ is either a P $n$-omino or L $n$-omino, then 
    $$a_{p,N} = \poly_{m+1}^n + r + \ceil{\frac rm}.$$
\end{theorem}

\begin{proof}
    The existence and uniqueness of $m$ and $r$ follow from \Cref{lem:polydecompose}.
    
    Let $p$ be a P $n$-omino, and let $q$ be an L $n$-omino. Then applying a row shift to $p$ followed by a column slide, $\sigma(\rho(p)) = q$. It follows from \Cref{lem:shiftorder} and \Cref{thm:slideorder} that $a_{q,N} \leq a_{p,N}$ for every positive integer $N$. By \Cref{lem:Lnominolowerbound} and \Cref{lem:Pnominoupperbound} ,
    $$\poly^n_{m+1} + r + \ceil{\frac rm} \leq a_{q,\poly^n_m + r} \leq a_{p,\poly^n_m + r} \leq \poly^n_{m+1} + r + \ceil{\frac rm}$$
    and therefore $a_{q,N} = a_{p,N} = \poly^n_{m+1} + r + \ceil{r/m}.$
\end{proof}

Recall from \Cref{sec:relatinginstancesequences},
$$a_{\,\sqmot,N} = a_{\,\zmot,N} \geq a_{\,\tttmo,N} = a_{\,\lmot,N}$$
for every positive integer $N$.
Since $\sqmo$ is a P 4-omino, and $\lmo$ is an L 4-omino, \Cref{thm:main} gives a formula for the instance sequences of all these tetrominoes. Since $\poly^4_m = m^2$, we arrive at the following.

\begin{corollary}\label{thm:tetromino_parameterized}
    For any positive integer N, there exists a unique positive integer $m$ and unique integer $r$ such that $0 \leq r \leq 2m$ and $N = m^2+r$. Furthermore, if $p$ is $\sqmo$, $\zmo$, $\ttmo$, or $\lmo$ then
    $$a_{p,N} = (m+1)^2 + r + \ceil{\frac rm}.$$
\end{corollary}

By similar reasoning, and in reference to \Cref{fig:posets}, we arrive at an analogous result for the L, N, P, W, and Y pentominoes.

\begin{corollary}\label{thm:Ppentomimo_parameterized}
    For any positive integer N, there exists a unique positive integer $m$ and unique integer $r$ such that $0 \leq r \leq 3m$ and $N = \poly^5_m+r$. Furthermore, if $p$ is $\lpmo$,$\smo$, $\ppmo$, $\wmo$, or $\ypmo$  then
    $$a_{p,N} = \poly^5_{m+1}+r + \ceil{\frac rm}.$$
\end{corollary}

\section{The F, T, V, X, and Z Pentominoes}\label{sec:otherpentominoes}

In this section, we find a formula for $a_{p,N}$ where $p$ is any of the pentominoes $\xmo$, $\zpmo$, $\tpmo$, $\fmo$, or $\vpmo$. We employ the same method that we used to prove \Cref{thm:trominomain} and  \Cref{thm:main}, finding an upper bound for $a_{\,\txmo,N}$ and $a_{\,\tfmo,N}$ which equals a lower bound for $a_{\,\tvpmo,N}$, and using the fact that
$$\ a_{\,\txmo,N} = a_{\,\tzpmo,N} \geq a_{\,\ttpmo,N} \geq a_{\,\tvpmo,N} \quad \text{ and } \quad a_{\,\tfmo,N} \geq a_{\,\ttpmo,N} \geq a_{\,\tvpmo,N}$$
for every positive integer $N$, as shown in \Cref{fig:posets}.

The \bfb{centered square numbers}, given by the sequence $s_m = m^2 + (m-1)^2$, play a central role in this section. Conway and Guy include these in their discussion of polygonal numbers in \textit{The Book of Numbers} \cite{Conway}. We will make use of the following lemma.

\begin{lemma}\label{lem:Xmorecursion}
    For any positive integer $N$, there exist a unique positive integer $m$ and unique integer $r$ such that $0 \leq r < 4m$ and $N = s_m + r$.
\end{lemma}

\begin{proof}
    This follows immediately from the fact that the centered square numbers are an increasing sequence of integers starting from $s_1=1$, and from the fact that $s_{m+1} - s_m = 4m$.
\end{proof}

\subsection{A Construction for X Pentominoes}

\begin{figure}
    \centering
    \input{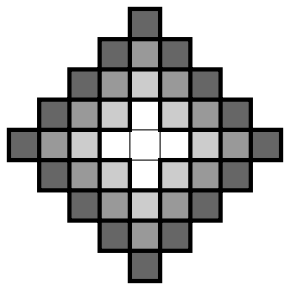}
    \caption{Multiples of  ${\protect\xmo}$ superimposed.}
    \label{fig: XMultiples}
\end{figure}

Several multiples of the X pentomino $\xmo$ are depicted in \Cref{fig: XMultiples}. The shape of $m \,\xmo$ is called the \bfb{Aztec diamond of order $m+1/2$} by Knuth \cite{Knuth}, which we denote $A_m$. The shape consists of all cells within $m$ orthogonal steps of some central cell.

If the cells of $A_m$ are colored like a checkerboard, the cells of each color form two diamond arrays, one having $(m+1)^2$ cells and the other having $m^2$ cells. Thus $|A_m| = s_{m+1}$. Indeed, this is where the centered square numbers get their name.

\begin{lemma}
    For any positive integer $m$, the multiple $m\,\xmo$ has shape $A_m$.
\end{lemma}

\begin{proof}
    We proceed by induction on $m$. For our base case, $m=1$, note that $\xmo$ itself is $A_1$. We hypothesize that for some integer $m > 1$, $(m-1)\,\xmo$ has shape $A_{m-1}$.

    We compute $m\,\xmo$ by taking the union of all instances of $\xmo$ with a center cell in some fixed instance of $(m-1)\,\xmo$. This union includes the cells of $(m-1)\,\xmo$ and also every cell that is one orthogonal step from any cell in $(m-1)\,\xmo$. Consequently, these are the cells within $m$ orthogonal steps of the central cell of $(m-1)\,\xmo$, which form the shape $A_{m}$. 
\end{proof}

    Each cell in $m\,\xmo$ corresponds to one instance of $\xmo$ in $(m+1) \xmo,$ so in general $A_m$ has $s_{m+1}$ cells and at least $s_m$ instances of $\xmo$.
    This shows that
    $$a_{\,\txmo,s_m} \leq s_{m+1}.$$

\begin{figure}
    \centering
    \input{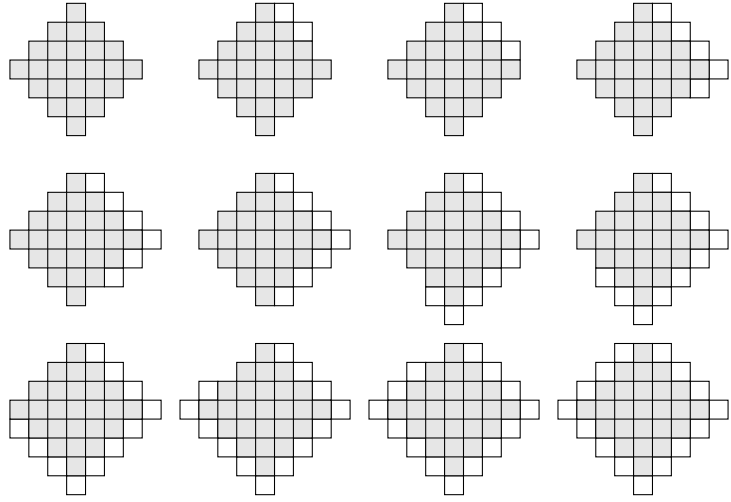}
    \caption{From left to right, top to bottom, these polyominoes respectively
contain $s_3=13$, 14, 15, 16, 17, 18, 19, 20, 21, 22, 23, $s_4-1=24$ instances of~${\protect \xmo}$.}
    \label{fig:SpiralDiamond}
\end{figure}
We can generalize this upper bound by adding cells around the perimeter of the shape, starting to the right of the top row and working clockwise, as in \Cref{fig:SpiralDiamond}.
The general upper bound is described by the following theorem.

\begin{lemma}\label{lem:Xmo_upper_bound}
    For any positive integer $m$ and integer $r$ such that $1 \leq r < 4m$,
    $$a_{\,\txmo,s_m} \leq s_{m+1}$$
    and
    $$a_{\,\txmo,s_m+r} \leq s_{m+1} + r + \floor{ \frac rm} + 1.$$
\end{lemma}

\begin{proof}
    We have already shown that $a_{\,\txmo,s_m} \leq s_{m+1}$ because $m\,\xmo$ has $s_{m+1}$ cells and at least $s_m$ instances of $\xmo$.
    To prove the second inequality, we will construct a polyomino $P$ with $s_{m+1} + r + \floor{ r/m} + 1$ cells having at least $s_m + r$ instances of $\xmo$. 
    
    We follow the construction depicted in \Cref{fig:SpiralDiamond}, adding cells to $mp$. We first add one cell to the right end of each row, working from the top down. We then add a new cell to the bottom of the center column of $mp$. Lastly we add a cell to the left end of each row of $mp$, working from the bottom up. We terminate the process as soon as $r$ new instances of $\xmo$ have been created.

    In this construction, we usually add one cell to create each new instance of $\xmo$, but we add two cells in order to create the 1st, $m$th, $2m$th, and $3m$th instances. The number of cells needed to create $r$ new instances of $\xmo$ is then $1+r + \floor{r/m}$, for a total of $s_{m+1} + r + \floor{ r/m} + 1$ cells and at least $s_m+r$ instances of $\xmo$. 
\end{proof}

\subsection{A Construction for F Pentominoes}

We give a construction for $(\fmo, N)$-dense polyominoes by modifying our construction from the previous section.

\begin{lemma}\label{lem:Fmobound}
    For any positive integer $m$ and integer $r$ such that $1 \leq r < 4m$,
    $$a_{\,\tfmo,s_m} \leq s_{m+1}$$
    and
    $$a_{\,\tfmo,s_m+r} \leq s_{m+1} + r + \floor{ \frac rm} + 1.$$
\end{lemma}

\begin{proof}
    We will construct a polyomino $Q$ with $s_{m+1}$ cells having at least $s_m$ instances of $\fmo$. The argument for $a_{\,\tfmo,s_m+r}$ uses identical reasoning.

    First, we construct a polyomino $P$ which has $s_{m+1}$ cells and at least $s_m$ instances of $\xmo$, as in \Cref{lem:Xmo_upper_bound}. Then, we shift each column to the right of the central column up one cell relative to the column to its left. We call the new polyomino $Q$. We will describe how each instance of $\xmo$ in $P$ corresponds to a unique instance of $\fmo$ in $Q$, as depicted in \Cref{fig:Fmo_construction}.

    Each instance of $\xmo$ having its four left-most cells to the left of the central column of $P$ has its same position in $Q$. It shares those four cells with an instance of $\fmo$ in $Q$.
    
    Each instance of $\xmo$ with three cells along the central column of $P$ is transformed cell-wise into an instance of $\fmo$ in $Q$.

    The remaining instances of $\xmo$ to the right of the central column of $P$ are transformed cell-wise into instances of $\zpmobackwards$ in $Q$. Each shares its four right-most cells with an instance of $\fmo$ in $Q$.
\end{proof}

\begin{figure}
    \centering
    \input{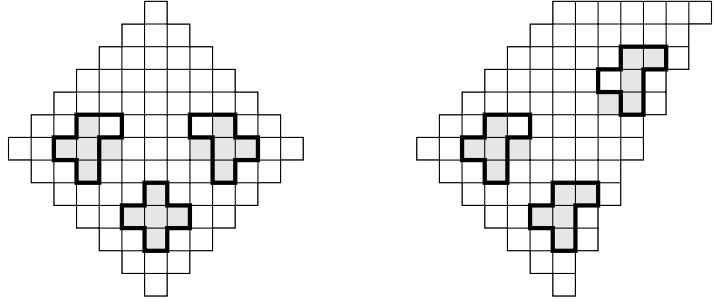}
    \caption{In both $P$  (left) and $Q$ (right), we shade the cells of an instance of ${\protect \xmo}$ from $P$, and draw the boundary around the corresponding instance of ${\protect \fmo}$ from $Q$. }
    \label{fig:Fmo_construction}
\end{figure}

\subsection{A Lower Bound for V Pentominoes}
The goal of this section is to give a lower bound on $a_{\,\tvpmo,N}$.
Within this section, we assume the coordinates of all cells are nonnegative, so that all cells rest in the first quadrant of the Cartesian plane.

Within this section we define the \bfb{diagonal} $D_i$, for each nonnegative integer $i$, to be the set of cells in the first quadrant with coordinates along the line $x+y = i$. (This is a special case of the definition of a diagonal used in \Cref{sec:lowerbounds}.) We define the \bfb{subdiagonals} $D^E_i$ and $D^O_i$ to be the subsets of $D_i$ of cells having even and odd $y$-coordinates respectively; we exclude $D^O_0$, which would be empty. We call the intersection of a subdiagonal with a polyomino a subdiagonal of that polyomino. We start with some technical lemmas about subdiagonals.

\begin{lemma}\label{lem:subdiagbound}
    For any positive integer $N$, if $P$ is a $(\vpmo,N)$-dense polyomino with $k$ subdiagonals, then $|P| \geq N+k$.
\end{lemma}

\begin{proof}
    Notice that $P$ has $N$ cells that are the top-most cell of an instance of $\vpmo$. 
    None of these can be the right-most cell of any subdiagonal of $P$, as the top-most cell and right-most cell of any instance of $\vpmo$ belong to the same subdiagonal.
    Therefore the top-most cells of instances of $\vpmo$, together with the right-most cells of each subdiagonal of $P$, form $N+k$ distinct cells in $P$.
\end{proof}

We may assume by \Cref{thm:corner-aligned} that the $(\vpmo, N)$-dense polyominoes we discuss are corner-aligned, and we may assume they have a cell with coordinates $(0,0)$.
Recall, we call such polyominoes \bfb{quadrant-aligned}.

\begin{lemma}\label{lem:diagorder}
    For any positive integer $N$, let $P$ be a quadrant-aligned $(\vpmo,N)$-dense polyomino. If $P$ intersects a subdiagonal, then $P$ intersects every subdiagonal of lower index.
\end{lemma}

\begin{proof}
    We proceed by induction on the index $k$ of the subdiagonals that intersect $P$. The base cases $k=0$ and $k=1$ follow directly from the quadrant-alignment of $P$.

    For our inductive hypothesis, suppose for some positive integer $k$ that if $P$ intersects a subdiagonal of index $k$ then it intersects every subdiagonal of lower index.
    For our inductive step, suppose $P$ intersects a subdiagonal of index $k+1$ at a cell with coordinates $(x,y)$. Note $x+y=k+1$.

    Consider the case that $x$ and $y$ are both positive. 
    Because $P$ is quadrant-aligned, it also contains the cells at coordinates $(x-1,y)$ and $(x,y-1)$, 
    which belong to the two separate subdiagonals of index $k$. The statement then follows by induction.

    Consider the case that $x \geq 2$ and $y=0$. (The case where $x = 0$ and $y \geq 2$ uses identical reasoning.) Since $P$ is $(\vpmo, N)$-dense, the cell $(x,y)$ belongs to some instance $p$ of $\vpmo$. The cell $(x,y)$ must be in the bottom row of $p$.
    
    If $(x,y)$ is the bottom left cell of $p$, then $P$ contains the cell $(x,y+1)$ which belongs to $p$, so $P$ also contains $(x-1,y+1)$ because $P$ is quadrant-aligned and $x$ is positive. 
    If $(x,y)$ is the bottom middle cell of $p$, then $P$ contains the cell $(x-1,y+1)$ which belongs to $p$. In both of these cases, the cell $(x-1,y+1)$ belongs to the same subdiagonal as $(x,y)$, and has positive coordinates. The statement then follows from the case with positive coordinates.

    If $(x,y)$ is the bottom right cell of $p$, then $P$ contains the cells $(x-1,y)$ and $(x-2,y+1)$ which belong to $p$; these belong to the two separate subdiagonals of index $k$. The statement then follows by induction.
\end{proof}

    \begin{figure}
        \centering

        \input{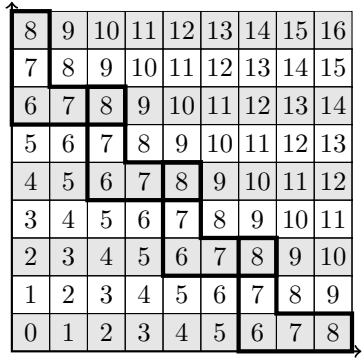}
        
        \caption{Each cell in the first quadrant is labeled with the index of its diagonal. The even subdiagonals are shaded gray. The $k$th diagonal $D_k$ contains $k+1$ cells. Note that the top- and right-most cells of any instance of ${\protect \vpmo}$ in the quadrant are consecutive cells of a subdiagonal. The number of instances of ${\protect \vpmo}$ that have a top- or right-most cell in a subdiagonal is then one less than the number of cells in that subdiagonal, as shown here for $D^E_8$ which has five cells and intersects four instances of ${\protect \vpmo}$ at their top- and right-most cells.}
        \label{fig:Vdiagonal}
    \end{figure}

\begin{lemma}\label{lem:subdiagcontain}
    For any nonnegative integer $M$, let $P$ be the union of $M$ subdiagonals, among which the highest index is $k$, such that every subdiagonal of index less than $k$ is included in the union.
    There exist unique integers $m$ and $u$ such that $M = 4m+u$ and $0 \leq u \leq 3$. Furthermore, $P$ has at most $s_m$ instances of $\vpmo$ for $u=0$, and at most $s_m+mu-1$ instances of $\vpmo$ for $u=1, 2, 3$.
\end{lemma}

\begin{proof}
    The existence and uniqueness of $m$ and $u$ follow from the quotient-remainder theorem. It may be helpful to refer to \Cref{fig:Vdiagonal}.

    If $M$ is odd, then $P$ is simply the union of the diagonals $D_0$, $D_1$,..., $D_k$. Then $P$ is a triangular array of shape $k \,\tmo$, which has exactly $t_{k-1}$ instances of~$\vpmo$.

    If $M$ is even, then $P$ only has one subdiagonal of index $k$, which we call $D'_k$.  Since $D_k$ has $k+1$ cells, $D'_k$ has at most  $\ceil{({k+1})/{2}}$ cells. Because no cell in $D'_k$ is immediately below or left of a cell in $P$, any instance of $\vpmo$ in $P$ only shares its top- or right-most cells with $D'_k$.
    The total number of instances of $\vpmo$ in the first quadrant that have a top- or right-most cell in $D'_k$ is one less than the maximum number of cells in $D'_k$, and is therefore at most $\ceil{({k+1})/{2}} - 1$; see \Cref{fig:Vdiagonal}. The number of instances of $\vpmo$ in the union of the subdiagonals of index less than $k$ is $t_{k-2}$. Therefore the number of instances of $\vpmo$ in $P$ is at most $t_{k-2} + \ceil{({k+1})/{2}} - 1$.

    Since $M$ is even, either $u=0$ or $u=2$; we consider these cases separately. If $u=0$ then $M=4m$ and $k = 2m$. Then the number of instances of $\vpmo$ in $P$ is at most
    $$t_{2m-2} + \ceil{({2m+1})/{2}} - 1 = (2m-1)(2m-2)/2 + m =  s_m.$$

    If $u=2$ then $M = 4m+2$ and $k = 2m+1$. Then the number of instances of $\vpmo$ in $P$ is at most
    $$t_{2m-1} + \ceil{({2m+2})/{2}} - 1 = 2m(2m-1)/2 + m+1 -1 = s_m + 2m-1.$$
\end{proof}

\begin{lemma}\label{lem:XLowerbound}
    For any positive integer $m$ and integer $r$ such that $1 \leq r < 4m$,
    $$a_{\,\tvpmo,s_m} \geq s_{m+1}$$
    and
    $$a_{\,\tvpmo,s_m+r} \geq s_{m+1} + r + \floor{ \frac rm} + 1.$$
\end{lemma}

\begin{proof}
    It follows from \Cref{thm:corner-aligned} that for any positive integer $N$ there exists a $(\vpmo, N)$-dense polyomino that is quadrant-aligned. By \Cref{lem:diagorder} if such a polyomino intersects a subdiagonal, then that polyomino intersects every subdiagonal of lower index. 

    To prove the bound for $a_{\,\tvpmo,s_m}$, we consider $P$ a quadrant-aligned $(\tvpmo, s_m)$-dense polyomino.  As $P$ is a subset of the union of the diagonals it intersects, by \Cref{lem:subdiagcontain}, $P$ has at least $4m$ subdiagonals. It follows from  \Cref{lem:subdiagbound} that
    $|P| \geq s_m + 4m = s_{m+1}$, which shows $a_{\,\tvpmo,s_m} \geq s_{m+1}$.

    The argument for $a_{\,\tvpmo,s_m+r}$ is similar. Consider $P$ a quadrant-aligned $(\tvpmo, s_m+r)$-dense polyomino. The number instances of $\vpmo$ in $P$ is $s_m + m \cdot \frac{r+1}{m} - 1$, so $P$ must have at least $4m + \ceil{(r+1)/m}$ subdiagonals by \Cref{lem:subdiagcontain}. It follows from \Cref{lem:subdiagbound} that
    \begin{align*}
        |P| &\geq s_m + r + 4m + \ceil{\frac{r+1}{m}} = s_{m+1} + r + \ceil{\frac{r+1}{m}}.
    \end{align*}
    It remains only to check that $\ceil{(r+1)/m} = \floor{r/m}+1$ when $m$ and $r$ are positive integers. We omit this argument, which is straightforward.
\end{proof}

\subsection{Combining Upper and Lower Bounds}

We conclude with the main result for this section.

\begin{theorem}\label{thm:xmoparameterized}
Let $p$ be $\fmo$, $\xmo$, $\zpmo$, $\tpmo$, or $\vpmo$.
    For any positive integer $N$, there exist a unique positive integer $m$ and unique integer $r$ such that $0 \leq r < 4m$ and $N = s_m + r$. Furthermore, if $r = 0$ then
    $$a_{p,N} = s_{m+1}$$
    and if $r > 0$ then
    $$a_{p,N} = s_{m+1} + r + \floor{\frac rm} + 1.$$
\end{theorem}

\begin{proof}
    Let $N$ be given.
    The existence and uniqueness of $m$ and $r$ follow from \Cref{lem:Xmorecursion}. 
    As depicted in \Cref{fig:posets}, since $\rho(\xmo) = \zpmo$, $\sigma(\xmo) = \tpmo$ and $\sigma'(\tpmo) = \vpmo$,
    $$a_{\,\tzpmo,N} = \ a_{\,\txmo,N} \geq a_{\,\ttpmo,N} \geq a_{\,\tvpmo,N}.$$ 
    Since $\sigma(\fmo) = \tpmo$,
    $$a_{\,\tfmo,N} \geq a_{\,\ttpmo,N} \geq a_{\,\tvpmo,N}.$$
    The theorem then follows directly from \Cref{lem:Xmo_upper_bound}, \Cref{lem:Fmobound}, and \Cref{lem:XLowerbound}: 
    In the case $r=0$, these lemmas give $s_{m+1}$ as an upper bound to $a_{\,\txmo,N}$ and $a_{\,\tfmo,N}$, and a lower bound to $a_{\,\tvpmo,N}$;
    In the case $r > 0$, these lemmas give $s_{m+1} + r + \floor{r/m} + 1$ as an upper bound to $a_{\,\txmo,N}$ and $a_{\,\tfmo,N}$, and a lower bound to $a_{\,\tvpmo,N}$.
\end{proof}

\section{Closed Forms of Instance Sequences}\label{sec:closedforms}

In this section, we describe closed forms of the instance sequences we have so far given parameterized forms. The closed forms are displayed in \Cref{tab: closed forms}.

    \begin{table}
      \centering
      \begin{tabular}{ll}
        $p$ & $a_{p,N}$ \\[5pt]
        \midrule
        $\tmo$

        & $\ceil{N+\sqrt{2N}+1}$ \\[7pt]

        $\sqmo$,\ $\zmo$,\ $\ttmo$,\ $\lmo$
        
        & $\ceil{N+\sqrt{4N}+1}$ \\[7pt]
        
        $\lpmo$,\ $\smo$,\ $\ppmo$,\ $\lwmo$,\ $\ypmo$
        
        & $\ceil{N+\sqrt{6N+\frac{1}{4}}+\frac{3}{2}}$ \\[7pt]
 
        $\lfmo$,\ $\lxmo$,\ $\lzpmo$,\ $\ltpmo$,\ $\lvpmo$

        & $\ceil{N+\sqrt{8N-4}+2}$ \\[7pt]
        \bottomrule
      \end{tabular}
      \caption{A closed form for $a_{p,N}$ for polyominoes with up to five cells, except for straight polyominoes and the U pentomino ${\protect\umo}$. For $p$ any straight polyomino, $a_{p,N} = |p|+N-1$. \cite{condon2024}}
      \label{tab: closed forms}
    \end{table}

\subsection{Closed Form for L Trominoes}
To prove \Cref{thm:tmo_upper_bound}, we showed that $a_{\,\tmot, t_{m}+r} \leq t_{m+1} + r+ \ceil{ r/m}$ by constructing polyominoes having $t_m+r$ instances of $\tmo$ and $t_{m+1} + r+ \ceil{ r/m}$ cells. By \Cref{thm:trominomain}, we know that these constructions are dense. 

Note that these constructions can have any number of cells greater than 2, except one more than a triangular number. Therefore $(a_{\,\tmot, N})_{N = 1}^\infty$ is the \textit{complementary} increasing sequence to $(t_m+1)_{m=1}^\infty$ within $\mathbb Z^{\geq 3}$.

One formula for the complementary sequence to $(t_m)_{m=1}^\infty$ is $a_n =n+\ceil{\sqrt{2n}\ }$, due to Lambek and Moser \cite{Lambek}.
Adding 1 to each term in this sequence, we arrive at a closed form for $a_{\,\tmot, N}$.

\begin{theorem}
    For any positive integer $N$, $a_{\,\tmot, N} = \ceil{N+\sqrt{2N}+1}.$
\end{theorem}

\subsection{Other Closed Forms}
Each of the formulas given in this section can be found using the machinery developed by Lambek and Moser \cite{Lambek}, by first identifying the complementary increasing sequence of positive integers skipped. For conciseness and completeness, we give independent proofs.

\begin{theorem}
    For any positive integer N, if $p$ is $\sqmo$, $\zmo$, $\ttmo$, or $\lmo$ then
    $$a_{p,N} = \ceil{N + \sqrt{4N} + 1 }.$$
\end{theorem}

\begin{proof}
    Let $N$ be given.
    By \Cref{thm:tetromino_parameterized}, there exist a unique positive integer $m$ and integer $r$ such that $0 \leq r \leq 2m$ and $N = m^2+r$, and $a_{p,N} = (m+1)^2 + r + \ceil{r/m}$. 

    Let $f = \ceil{N + \sqrt{4N} + 1 }$. We must show $f = a_{p,N}$. Since $f = m^2+r + \ceil{\sqrt{4N}} + 1$ and $a_{p,N} = (m^2 + 2m + 1) + r + \ceil{r/m}$, it suffices to show that $\ceil{\sqrt{4N}} = 2m + \ceil{r/m}$. Note that
    $$4N = 4m^2 + 4r = \left\{\begin{array}{ll}
        (2m)^2 & \text{ if } r = 0,\\
        (2m+1)^2 & \text{ if } r=m+1/4,\\
        (2m+2)^2 & \text{ if } r=2m+1.\\
    \end{array}\right.$$
    As $r$ and $m$ are necessarily integers,
    $$\ceil{\sqrt{4N}} = \left\{\begin{array}{ll}
        2m & \text{ if } r = 0,\\
        2m+1 & \text{ if } 1 \leq r \leq m,\\
        2m+2 & \text{ if } m+1 \leq r \leq 2m,\\
    \end{array}\right.$$
    so $\ceil{\sqrt{4N}} = 2m + \ceil{r/m}$, as was to be shown.
\end{proof}

\begin{theorem}
    For any positive integer N, if $p$ is $\ppmo$, $\smo$ , $\ypmo$, $\lpmo$, or $\wmo$ then
    $$a_{p,N} = \ceil{N+\sqrt{6N+\frac{1}{4}}+\frac{3}{2}} $$
\end{theorem}

\begin{proof}
    Let $N$ be given. By \Cref{thm:Ppentomimo_parameterized}, there exist a unique positive integer $m$ and integer $r$ such that $0 \leq r \leq 3m$, and $N = \poly^5_m+r$, and $a_{p,N} = \poly^5_{m+1}+r + \ceil{r/m}$.

    Let $f = \ceil{N+\sqrt{6N+{1}/{4}}+{3}/{2}}$. We must show $f = a_{p,N}$. Since $f = \poly^5_m + r + \ceil{\sqrt{6N+{1}/{4}}+{1}/{2}}+1$ and $a_{p,N} = (\poly^5_m + 3m + 1) + r + \ceil{r/m}$ by  \Cref{lem:recursion}, it suffices to show that $\ceil{\sqrt{6N+{1}/{4}}+\frac{1}{2}} = 3m + \ceil{r/m}$. This can be checked in cases, using the fact that $\poly^5_m = (3m^2-m)/2$. Note that
    \begin{align*}
        \sqrt{6N+ \frac 14} + \frac 12 &= \sqrt{9m^2-3m +6r + \frac 14} + \frac 12\\
        &= 
        \left\{\begin{array}{ll}
        3m & \text{ if } r = 0,\\
        3m+1 & \text{ if } r = m,\\
        3m+2 & \text{ if } r = 2m + 1/3,\\
        3m+3 & \text{ if } r = 3m + 1.
        \end{array}\right..
    \end{align*}
    As $r$ and $m$ are necessarily integers,
    $$
    \ceil{\sqrt{6N+ \frac 14} + \frac 12} =
    \left\{\begin{array}{ll}
    3m & \text{ if } r = 0,\\
    3m+1 & \text{ if } 1 \leq r \leq m,\\
    3m+2 & \text{ if } m+1 \leq r \leq 2m,\\
    3m+3 & \text{ if } 2m+1 \leq r \leq 3m,\\
\end{array}\right.$$
so $\ceil{\sqrt{6N+{1}/{4}}+\frac{1}{2}} = 3m + \ceil{r/m}$ as was to be shown.
\end{proof}

\begin{theorem}For any positive integer $N$, if $p$ is $\fmo$, $\xmo$, $\zpmo$, $\tpmo$, or $\vpmo$ then
    $$a_{p,N} = \ceil{N+\sqrt{8N-4}+2}.$$
\end{theorem}

\begin{proof}
    Let $N$ be given. 
    Recall that $s_m = m^2 + (m-1)^2$, and $s_{m+1} = s_m + 4m$. 
    By \Cref{lem:Xmorecursion}, there exist a unique positive integer $m$  and unique integer $r$ such that $0 \leq r \leq 4m-1$ and $N = s_m+r$. We treat the cases $r > 0$ and $r = 0$ separately.

    If $r > 0$, then $N = s_m + r$ and $a_{p,N} = (s_{m}+4m) + r + \floor{r/m} + 1$. Let $f = \ceil{N+\sqrt{8N-4}+2}$. We must show $f = a_{p,N}$. Since $f = s_m  + r + \ceil{\sqrt{8N-4}}+2$, it suffices to show that $\ceil{\sqrt{8N-4}} = 4m+ \floor{r/m}-1$. Note that
    \begin{align*}
        \sqrt{8N-4} &= \sqrt{16m^2 - 16m+8r+4} = \left\{
        \begin{array}{ll}
            4m - 2 & \text{ if } r=0, \\
            4m-1 & \text{ if } r = m - 3/8,\\
            4m & \text{ if } r = 2m - 1/2,\\
            4m+1 & \text{ if } r = 3m - 3/8,\\
            4m+2 & \text{ if } r = 4m.
        \end{array}
        \right.
    \end{align*}  
    As $r$ and $m$ are necessarily integers,
    $$\ceil{\sqrt{8N-4}} = \left\{
    \begin{array}{ll}
        4m-1 & \text{ if } 1 \leq r \leq m-1 \\
        4m & \text{ if } m \leq r \leq 2m-1 \\
        4m+1 & \text{ if } 2m \leq r \leq 3m-1 \\
        4m+2 & \text{ if } 3m \leq r \leq 4m-1 \\
    \end{array}
    \right.$$
    so $\ceil{\sqrt{8N-4}} = 4m+ \floor{r/m}-1$ as was to be shown.

    In the case $r = 0$, $N = s_m$ and $a_{p,N} = s_{m+1} = s_m+4m$. As we computed above, $\sqrt{8N-4} = 4m-2$, and so
    $$\ceil{N + \sqrt{8N-4}+2} = \ceil{s_m + 4m-2 + 2}=s_{m+1} = a_{p,N}$$
    as was to be shown.
    \end{proof}

\subsection{Complements to Sequences of Skipped Numbers}
In a recent preprint \cite{Cloitre}, Benoit Cloitre defines an $S(x,y,z)$ sequence to be an increasing sequence of integers $a_k$ starting with $a_1 = x$, such that for $k > 1$, $a_{k}-a_{k-1} = y$ if $k$ occurs in the sequence before position $k$, and otherwise $a_{k}-a_{k-1} = z$. In particular, $S(x,1,2)$ usually increments by 1 but skips an integer at its $k$th position if and only if $k$ does not occur earlier in the sequence. 

Cloitre gives combinatorial descriptions of $S(3,1,2)$ and $S(4,1,2)$ which show these are the same as our instance sequences for $\tmo$ and $\lmo$ respectively. We believe $S(5,1,2)$ is the same as the instance sequence for $\lpmo$, and we suspect that the instance sequence for the L $n$-omino is $S(n,1,2)$ in general.

The instance sequence for $\xmo$ is not an $S(x,y,z)$ sequence, but it appears to have analogous behavior to such a sequence after its second term.

\section{Instance Sequences are Asymptotically Linear}\label{sec:instancesequences}

In this section, we show that the instance sequence $(a_{p,N})_{N=1}^\infty$ for any polyomino $p$ must be asymptotically linear, by providing upper and lower bounds on $a_{p,N+1} - a_{p,N}$ which are independent of $N$.

\begin{lemma}\label{lem: linear lower bound}
    For any polyomino $p$ and positive integer $N$, $a_{p,N} < a_{p,N+1}$.
\end{lemma}

\begin{proof}
    Let $P_{N+1}$ be any $(p,N+1)$-dense polyomino. By definition $P$ contains at least $N+1$ instances of $p$. It suffices to show that there exists a set of cells smaller than $P_{N+1}$ that has at least $N$ instances of $p$; by \Cref{lem:connected}, there then exists a polyomino having at least $N$ instances of $p$ that is smaller than $P_{N+1}$.
    
    We remove the top cell $c$ of the left-most column of $P_{N+1}$. If $c$ belongs to any instance of $p$ in $P_{N+1}$, then it must be the top cell of the left-most column of that instance of $p$. Therefore $c$ belongs to at most one instance of $p$ in $P_{N+1}$. The removal of $c$ therefore creates a set of cells smaller than $P_{N+1}$, having at least $N$ instances of $p$, as was to be shown.
\end{proof}

Note that this implies $1 \leq a_{p,N+1} - a_{p,N}$. We now give an upper bound.

\begin{lemma} \label{lem: linear upper bound}
    For any polyomino $p$ and positive integer $N$, $$a_{p,N+1} - a_{p,N} \leq a_{p,2} - a_{p,1}.$$
\end{lemma}

\begin{proof}
    Let $P_2$ be any $(p,2)$-dense polyomino. Then $P_2$ can be thought of as the union of $p'$ some instance of $p$ and a translation of $p'$; call this translation $T$, so that $P_2$ is the union of $p'$ and $T(p')$. 
    Note that the number of cells in $T(p')$ which are not in $p'$ is $a_{p,2} - a_{p,1}$. 
    Without a loss of generality we assume $T$ moves the polyomino to the left, as well as possibly up or down.

    Let $P_N$ be any $(p,N)$-dense polyomino, and let $\overline p$ be any of the left-most instances of $p$ in $P_N$. Since $T(\overline p)$ is to the left of $\overline p$, it is distinct from any instance of $p$ in $P_N$. Furthermore, $T(\overline p)$ contains at most $a_{p,2}-a_{p,1}$ cells which do not already belong to both $\overline p$ and $P_N$.

    Therefore the union of $P_N$ and $T(\overline p)$ has at least $N+1$ instances of $p$, and contains at most $a_{p,N} + (a_{p,2}-a_{p,1})$ cells. Hence
    $$a_{p,N+1} \leq a_{p,N} + (a_{p,2} - a_{p,1}),$$
    from which the lemma follows immediately.
\end{proof}

It follows from \Cref{lem: linear lower bound,lem: linear upper bound} that
$$a_{p,1} + (N-1) \leq a_{p,N} \leq a_{p,1} + (N-1)(a_{p,2} - a_{p,1})$$
for any polyomino $p$ and any positive integer $N$, establishing the following theorem.

\begin{theorem}
    The instance sequence $(a_{p,N})_{N=1}^\infty$ of any polyomino $p$ is asymptotically linear.
\end{theorem}

\section{Novel Prismatic Polyominoes}\label{sec:novelprismatic}

In this section we present novel prismatic polyominoes. Several of these were found through a branch-and-prune search, and are depicted in \Cref{fig: novel prismatic polyominoes}. It is straightforward to check, using the formulas in \Cref{tab: closed forms}, that these are dense as well as de Bruijn.

\begin{figure}
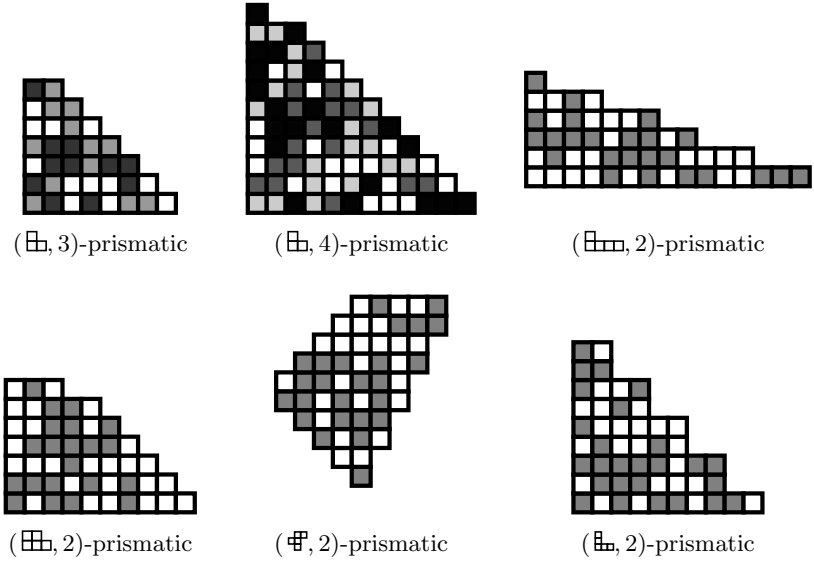

\begin{subfigure}{0.25\textwidth}
    \centering
    \input{Images/L_tromino_3_coloring}{}
    \caption*{$(\,\tmo,3)$-prismatic}
    \label{fig:L tromino 3 coloring}
    \end{subfigure}
\begin{subfigure}{0.3\textwidth}
        \centering
    \input{Images/L_tromino_4_coloring}{}
    \caption*{$(\,\tmo,4)$-prismatic}
    \label{fig:L tromino 4 coloring}
    \end{subfigure}
\begin{subfigure}{0.35\textwidth}
      \centering
    \input{Images/L_pentomino_2_coloring}{}
    \caption*{$(\,\lpmo,2)$-prismatic}
    \label{fig:L pentomino 2 coloring}
\end{subfigure}\\

\begin{subfigure}{0.25\textwidth}
     \centering
    \input{Images/P_pentomino_2_coloring}
    \caption*{$(\,\ppmo,2)$-prismatic}
    \label{fig:P pentomino 2 coloring}
\end{subfigure}
\begin{subfigure}{0.3\textwidth}
     \centering
    \input{Images/F_pentomino_2_coloring}
    \caption*{$(\,\fmo,2)$-prismatic}
    \label{fig:F pentomino 2 coloring}
\end{subfigure}
\begin{subfigure}{0.35\textwidth}
     \centering
    \input{Images/V_pentomino_2_coloring}
    \caption*{$(\,\vpmo,2)$-prismatic}
    \label{fig:V pentomino 2 coloring}
\end{subfigure}
\caption{These prismatic polyominoes were found using a branch-and-prune search. With a row shift, the $({
\protect\lpmo},2)$-prismatic polyomino can be transformed into a $({\protect \ypmo},2)$-prismatic polyomino. The $({\protect\ppmo},2)$-prismatic polyomino can similarly be transformed into an $({\protect \smo},2)$- or $({\protect \wmo},2)$-prismatic polyomino.}
\label{fig: novel prismatic polyominoes}
\end{figure}

We devote the remainder of this section to a method of constructing $(\xmo,2)$-prismatic polyominoes. With a row shift, these can be transformed into $(\zpmo,2)$-prismatic polyominoes.

The unique shape of a $(\sqmo,2)$-prismatic polyomino is a $5 \times 5$ square, though there are 800 possible colorings \cite{condon2024}. We say a $(\sqmo,2)$-prismatic polyomino is \bfb{toric} if its first and last rows are identical, and its first and last columns are identical; one can imagine the shape on a torus, with these rows and columns identified. Two such polyominoes are depicted in \Cref{fig:toric square solutions}. An exhaustive search shows that these are the only toric $(\sqmo,2)$-prismatic polyominoes, up to rotation, reflection, and color permutation.

\begin{figure}
\begin{subfigure}{0.3\textwidth}
    \centering
    \input{Images/Square_torus}
    \caption*{toric square $A$}
\end{subfigure}
\begin{subfigure}{0.3\textwidth}
    \centering
    \input{Images/Square_torus_shifted}
    \caption*{another toric square}
    \end{subfigure}
\centering
    \caption{Two toric $(\protect\sqmo,2)$-prismatic polyominoes.}
    \label{fig:toric square solutions}
\end{figure}

\begin{figure}
    \centering
    \input{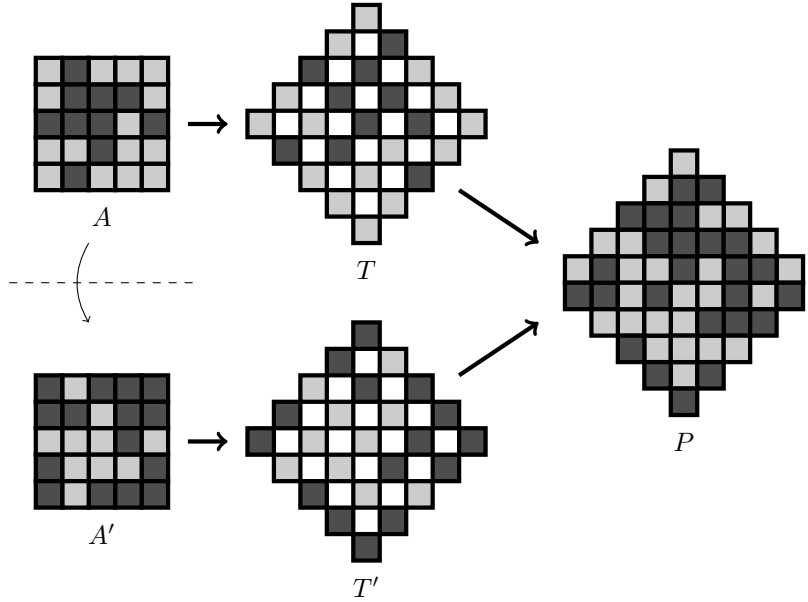}
    \caption{Constructing a $({\protect\xmo},2)$-prismatic polyomino. The white squares in this figure are holes in the shapes, not cells.}
    \label{fig:bandaidconstruction}
\end{figure}

We will start our construction with the left toric square from \Cref{fig:toric square solutions}, which we name $A$. The steps of the construction are depicted in \Cref{fig:bandaidconstruction}.
We reflect $A$ across a horizontal line, and swap its colors, to create the colored polyomino $A'$.

We rotate the positions of the cells in $A$ by 45 degrees clockwise, not rotating the cells themselves, while expanding the shape so that the cells fit into alternating positions of an Aztec diamond of order $4+1/2$. We call this new colored shape $T$.  We perform the same operation on $A'$, producing a colored shape we call $T'$.

Lastly, we overlay $T$ and $T'$, so that the shape of $T$ is translated one cell above from the shape of $T'$. We call the final colored polyomino $P$. We note that $P$ has 50 cells and is $(\xmo,32)$-dense. 

There are 16 ways to color the four outer cells of $\xmo$ with two colors. In this construction, each instance of $\sqmo$ in $A$ and $A'$ is translated cell-wise to the four outer cells of a different instance of $\xmo$. Thus each of the ways of coloring the four outer cells of $\xmo$ occurs twice in $P$.
The coloring is de Bruijn because the center cells of these two instances always differ in color when one starts with the toric square $A$ from \Cref{fig:toric square solutions}. This process also generates $({\protect \xmo}, 2)$-prismatic polyominoes when one starts with the toric squares shown in \Cref{fig:bandaidalternates}, though it does not work for all toric squares. We can offer little insight into why the process works for just these particular squares.

\renewcommand{\arraystretch}{1.5}
\begin{figure}
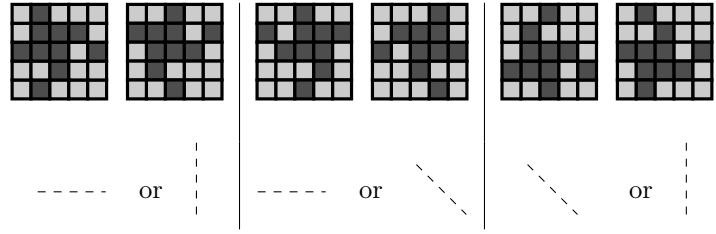

    \centering
    \begin{tabular}{c|c|c}
        \input{Images/Torus_1_1a} \ \input{Images/Torus_1_8a} & \input{Images/Torus_1_2a} \  
        \input{Images/Torus_1_7a} & \input{Images/Torus_1_4a} \ \input{Images/Torus_1_5a} \\[10pt]
        \input{Images/Torus_axes} & \input{Images/Torus_axes_2} & \input{Images/Torus_axes_3}
    \end{tabular}
    \caption{The process from \Cref{fig:bandaidconstruction} produces an $({\protect \xmo}, 2)$-prismatic polyomino starting with any of these squares as $A$, and reflecting over either of the indicated axes to produces $A'$.}
    \label{fig:bandaidalternates}
\end{figure}

\section{Prismatic Polyforms}\label{sec:polyforms}

A \bfb{polyform} is a connected shape made from identical regular polygonal cells glued together edge-to-edge. Polyominoes are polyforms with square cells, which we take to be faces on the square lattice. Similarly, \bfb{polyhexes} are polyforms with hexagonal cells, which we take to be faces on the hexagonal lattice. \bfb{Polypents} are polyforms with pentagonal cells, which we take to be faces of the dodecahedron.

\subsection{Polyhexes}

Many of our results for polyominoes can be extend to polyhexes. We define density and prismatic-ness analogously for polyhexes as for polyominoes. We let $\tau$ be the transformation on sets of square cells that moves each row right by half a cell relative to the row above it, and then replaces each square cell with a hexagonal cell. We call this the \bfb{polyhex shift}. See \Cref{fig: polyhex transformation}.

\begin{figure}
    \centering
    \resizebox{12cm}{!}{\input{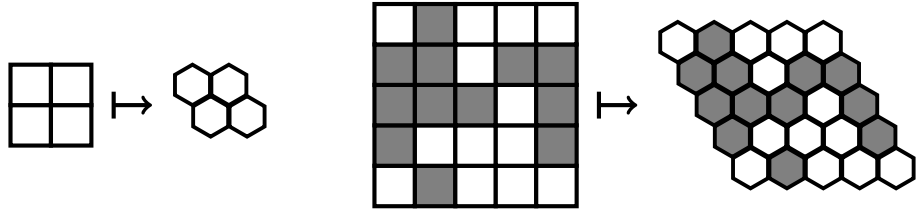}}
    \caption{The polyhex shift $\tau$ maps polyominoes to polyhexes. It maps prismatic polyominoes to prismatic polyhexes.}
    \label{fig: polyhex transformation}
\end{figure}

Note that $\tau$ preserves edge-to-edge connections between cells. Thus for any polyomino $p$, its image $\tau(p)$ is a polyhex. While the converse is not true, there is still as strong connection between dense polyominoes and dense polyhexes.

\begin{lemma}
    For any polyomino $p$ and positive integer $N$, if $P$ is a $(p,N)$-dense polyomino then $\tau(P)$ is a $(\tau(p),N)$-dense polyhex.
\end{lemma}

\begin{proof}
    Each instance of $p$ in $P$ is transformed cell-wise into an instance of $\tau(p)$ in $\tau(P)$. 
    Therefore $\tau(P)$ is a polyhex of size $|P|$ having at least $N$ instances of~$\tau(p)$.
    
    If $Q$ is any set of hexagonal cells having at least $N$ instances of $\tau(p)$,
    then each instance of $\tau(p)$ in $Q$ is transformed cell-wise into an instance of $p$ in $\tau^{-1}(Q)$. Thus $\tau^{-1}(Q)$ is a set of square cells having at least $N$ instances of $p$; by \Cref{lem:connected}, $|P| \leq \abs{\tau^{-1}(Q)}$. Since
    $|\tau(P)| = |P| \leq |\tau^{-1}(Q)| = |Q|$, it follows that $\tau(P)$ is $(\tau(p), N)$-dense.
\end{proof}

Because $\tau$ acts cell-wise, it induces a transformation of colored polyominoes, as depicted in \Cref{fig: polyhex transformation}. If $P$ is $(p,n)$-de Bruijn, then each distinct coloring of $p$ in $P$ maps to a distinct coloring of $\tau(p)$ in $\tau(P)$, so that $\tau(P)$ is $(\tau(p), n)$-de Bruijn. Therefore, $\tau$ maps prismatic polyominoes to prismatic polyhexes.

Some polyhexes are not the image of any polyomino under $\tau$, so their prismatic polyhexes cannot be constructed from prismatic polyominoes in this way. For example, $\tfidget$ and $\cthex$ are polyhexes for which $\tau^{-1}(\tfidget)$ and $\tau^{-1}(\cthex)$ are disconnected sets of square cells. In \Cref{fig:fidget spinner} we present de Bruijn polyhexes for these shapes which we conjecture are dense and therefore prismatic.

\begin{figure}
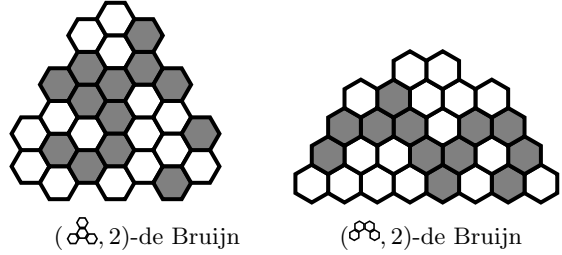

    \centering
    \begin{subfigure}{0.3\textwidth}
    \input{Images/Fidget_spinner_2_coloring}
    \caption*{$(\protect\tfidget, 2)$-de Bruijn}
     \end{subfigure}
     \begin{subfigure}{0.3\textwidth}
         \rotatebox{90}{\input{Images/Ctetrahex2coloring }}
         \caption*{{$(\protect\cthex,2)$-de Bruijn}}
     \end{subfigure}
    \caption{Two de Bruijn polyhexes we conjecture are prismatic. }
    \label{fig:fidget spinner}
\end{figure}

\subsection{Polypents}

On the dodecahedron, three faces that share a vertex form a polyform we call a \bfb{tripent}. We note that there are twenty vertices on the dodecahedron, and therefore twenty tripents. There are also twenty distinct ways of coloring the faces of a tripent with up to four colors, up to rotations and reflections: four colorings with just one color, twelve colorings with two colors, and four colorings with three colors.
The dodecahedron with the net depicted in \Cref{fig:dodecahedron} is de Bruijn in the sense that it includes each of these colorings exactly once.

\begin{figure}
    \centering
    \input{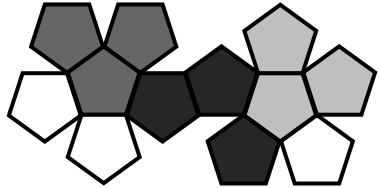}
    \caption{The net of a prismatic dodecahedron for tripents.}
    \label{fig:dodecahedron}
\end{figure}

\section{Conclusion}\label{sec:conclusion}

We have given explicit formulas for the instance sequences of all trominoes, tetrominoes, and pentominoes except for the U pentomino $\umo$. We note that this is the only polyomino with up to five cells that does not have both convex rows and convex columns. 
Our results imply that the instance sequence for $\umo$ is bounded between those of $\ppmo$ and $\xmo$.
We conjecture that $a_{\,\tumo,N} = a_{\,\xmo,N}$.

Polyomino addition was central to how we computed instance sequences, both in constructing upper bounds and implicitly in finding lower bounds. In all these cases, the multiple of a polyomino $mp$ is $(p, f(m))$-dense for $f$ some increasing function of $m$. We conjecture that the same is true for a broad category of polyominoes.

Efficient methods for generating de Bruijn colorings are not yet known for most $(p,n^{|p|})$-dense polyominoes. In fact, it is not yet known whether such polyominoes always admit a de Bruijn coloring; it remains an open question whether a $(p,n)$-prismatic polyomino exists for every polyomino $p$ and positive integer $n$.

\subsection*{Acknowledgments}

This work was conducted as part of the 2024 SMALL REU at Williams College under grant DMS2241623 from the National Science Foundation. Additional funding was provided by Harvey Mudd College and Amherst College. We thank these institutions for their support, and we thank Steve Miller especially for organizing the 2024 SMALL REU.

\printbibliography

\end{document}